\documentclass[11pt]{amsart}
\usepackage{amscd,amssymb,graphics}
\usepackage{hyperref}
\usepackage{amsfonts}
\usepackage{amsmath}
\usepackage{enumerate}
\input xy
\xyoption{all}

\newtheorem{thm}{Theorem}[section]
\newtheorem{fact}[thm]{Fact}
\newtheorem{corol}[thm]{Corollary}
\newtheorem{lemma}[thm]{Lemma}

\newtheorem{quest}[thm]{Question}
\newtheorem{defi}[thm]{Definition}

\theoremstyle{remark}
\newtheorem{remark}[thm]{Remark}
\newtheorem{example}[thm]{Example}

\newcommand{\ben}{\begin{enumerate}}
\newcommand{\een}{\end{enumerate}}
\newcommand{\bit}{\begin{itemize}}
\newcommand{\eit}{\end{itemize}}

\def\R {{\Bbb R}}
\def\Q {{\Bbb Q}}

\def\nbd{{neighborhood}}

\def\N{{\Bbb N}}
\def\T{{\Bbb T}}

\def\Z {{\Bbb Z}}
\def\U{{\Bbb U}}

\def\Homeo{{\mathrm{Homeo}}\,}

\def\eps{{\varepsilon}}
\def\K {\mathcal K}

\def\QED{\nobreak\quad\ifmmode\roman{Q.E.D.}\else{\rm Q.E.D.}\fi}

\def\a {\alpha}

\def\sna{\scriptscriptstyle\mathcal{NA}}

\def\Unif{\operatorname{Unif}}
\def\Comp{\operatorname{Comp}}
\def\Norm{\operatorname{Norm}}
\def\Conf{\operatorname{Conf}}
\begin{document}

\title[]{Free non-archimedean topological groups}
\author[]{Michael Megrelishvili}
\address{Department of Mathematics,
Bar-Ilan University, 52900 Ramat-Gan, Israel}
\email{megereli@math.biu.ac.il}
\urladdr{http://www.math.biu.ac.il/$^\sim$megereli}

\author[]{Menachem Shlossberg}
\address{Department of Mathematics,
Bar-Ilan University, 52900 Ramat-Gan, Israel}
\email{shlosbm@macs.biu.ac.il}
 \urladdr{http://www.macs.biu.ac.il/$^\sim$shlosbm}

\date{May 13, 2013}

\begin{abstract} We study
 free topological groups defined over uniform spaces in some subclasses of the class $\mathbf{NA}$ of non-archimedean groups.
 Our descriptions of the corresponding topologies
 show that for metrizable uniformities the corresponding free balanced, free abelian and free Boolean $\mathbf{NA}$
 groups are also metrizable. Graev type ultra-metrics determine the corresponding free topologies.
 Such results are in a striking contrast with free balanced and free abelian topological groups cases (in standard varieties).

 Another contrasting advantage is that the induced topological group actions on free abelian $\mathbf{NA}$
 groups frequently remain continuous.
One of the main applications is: any epimorphism in the category
$\mathbf{NA}$ must be dense. Moreover, the same methods improve the following
result of T.H. Fay \cite{Fay}: the inclusion of a proper open
subgroup $H \hookrightarrow G \in \mathbf{TGR}$ is not an
epimorphism in the category $\mathbf{TGR}$ of all Hausdorff
topological groups. A key tool in the proofs is Pestov's test of
epimorphisms \cite{Pest-epic}.

Our results provide a convenient way to produce surjectively universal $\mathbf{NA}$ abelian and balanced
groups. In particular, we unify and strengthen some recent results of Gao \cite{GAO} and Gao-Xuan \cite{Gao-Xuan}
 as well as classical  results about profinite groups
 which go back to Iwasawa and Gildenhuys-Lim \cite{GL}.
\end{abstract}

\maketitle

\setcounter{tocdepth}{1}

 \tableofcontents

\section{Introduction and preliminaries}
\label{s:intro}

\subsection{Non-archimedean groups and uniformities}
A topological group $G$ is said to be {\it non-archimedean} if it
has a local base $B$ at the identity consisting of open subgroups.
Notation: $G \in \mathbf{NA}$.
 If in this definition every $H \in B$ is a normal subgroup of $G$ then we obtain the subclass of all \emph{balanced} (or, $\mathbf{SIN}$) non-archimedean groups.
 Notation: $G \in \mathbf{NA_b}$.
 All prodiscrete (in particular, profinite) groups are in $\mathbf{NA_b}$.

 A uniform space is called {\it non-archimedean} if it possesses a base of
equivalence relations. Observe that a topological group is
non-archimedean if and only if its left (right) uniform structure is
non-archimedean.
 The study of non-archimedean groups and non-archimedean uniformities  has great influence on
 various fields of Mathematics: Functional Analysis,
 Descriptive Set Theory and Computer Science are only some of them.
  The reader can get a general impression  from  \cite{ro, bk, Les86, Lem03, MS1} and references therein.

\subsection{Free groups in different contexts} Recall that
according to \cite{Mar} any continuous map from a Tychonoff space
$X$ to a topological group $G$ can be uniquely extended to a continuous
homomorphism from the (Markov) free topological group $F(X)$ into
$G.$ Moreover, $X$ is a (closed) topological subspace of $F(X).$
 There are several descriptions of free topological
groups. See for example, \cite{Tk, pes85, Us-free, Sip}. Considering
 the category of uniform spaces and uniformly continuous
maps one obtains the definition of a {\it uniform free topological
group} $F(X,{\mathcal U})$ (see \cite{Num2}). A description
 of the topology of this group was given by Pestov \cite{pes85, Pest-Cat}.
\emph{Free topological $G$-groups}, the $G$-space version of the
above notions, were introduced in \cite{Me-F}.

Let $\Omega$ be a class of some Hausdorff topological groups.
We study in Section \ref{sec:ufna} a useful unifying concept of the \emph{$\Omega$-free topological groups.}


\begin{remark} \label{r:diag}
(`Zoo' of free $\mathbf{NA}$ groups)
Here we give a list of some natural subclasses $\Omega$ of $\mathbf{NA}$
and establish the notation for the corresponding free groups. 
These groups are well defined by virtue of Theorem \ref{thm:Samuel}.  
\ben
\item $\Omega =\mathbf{NA}$. The free non-archimedean group
$F_{\sna}$.
\item
 $\Omega =\mathbf{AbNA}$. The \emph{free non-archimedean abelian group} $A_{\sna}$.
\item $\Omega =\mathbf{NA_b}$. The \emph{free non-archimedean balanced group}
$F^b_{\sna}$.
\item $\Omega =\mathbf{BoolNA}$. The \emph{free non-archimedean Boolean group} $B_{\sna}$.
\item $\Omega =\mathbf{NA} \cap \textbf{Prec}$. The \emph{free non-archimedean precompact group}
$F^{\scriptscriptstyle Prec}_{\sna}$.
\item $\Omega = \textbf{Pro}$. The \emph{free profinite group}
$F_{\scriptscriptstyle Pro}.$
\item $\Omega = \textbf{BoolPro}$. The \emph{free Boolean profinite group} $B_{\scriptscriptstyle Pro}$.
\een


The following diagram demonstrates the interrelation (by the induced homomorphisms) between these free groups
defined over the same 
uniform space $(X,\mathcal{U})$.

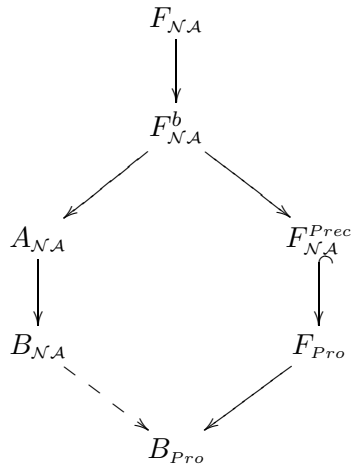
\begin{equation*}
\xymatrix{
& F_{\sna} \ar[d] &  \\
& F_{\sna}^b \ar[dl]  \ar[dr] &  \\
A_{\sna} \ar[d] & & F^{\scriptscriptstyle Prec}_{\sna} \ar@{^{(}->}[d] \\
B_{\sna} \ar@{-->}[dr] & & F_{\scriptscriptstyle Pro} \ar[dl] \\
& B_{\scriptscriptstyle Pro} &
}
\end{equation*}

 $F^{\scriptscriptstyle Prec}_{\sna} \hookrightarrow F_{\scriptscriptstyle Pro}$
is the completion of the group $F^{\scriptscriptstyle Prec}_{\sna}$ and
$B_{\sna} \dashrightarrow B_{\scriptscriptstyle Pro}$ is a dense injection. Other arrows are onto.

\end{remark}

We give descriptions of the topologies of these groups in Sections \ref{s:final} and \ref{s:pro}.
These descriptions show that for metrizable uniformities the corresponding free balanced, free abelian and free Boolean non-archimedean groups
 are also metrizable. The same is true (and is known)
 for the free profinite group which can be treated as the free compact non-archimedean group over a uniform space.
 Such results for the subclasses of $\mathbf{NA}$ are in a striking contrast with the standard classes outside of $\mathbf{NA}$.
 Indeed, it is well known that the usual free topological and free abelian topological groups $F(X)$ and $A(X)$ respectively,
 are metrizable only for discrete topological spaces $X$. Similar results are valid for uniform spaces.

In Section \ref{s:H} we discuss the free Boolean profinite group $B_{\scriptscriptstyle Pro}(X)$ of
a Stone space $X$ which is the Pontryagin dual of the discrete Boolean group
of all clopen subsets in $X$.

 In Section \ref{s:co-un} we unify
 and strengthen some recent results of Gao \cite{GAO} and Gao-Xuan \cite{Gao-Xuan}
 about the existence and the structure of surjectively universal non-archimedean Polish groups for abelian and balanced cases;
 as well as,
 results on surjectively universal profinite groups which go back to Iwasawa and Gildenhuys-Lim \cite{GL}.

\subsection{The actions which come from automorphisms}

 Every continuous group action of $G$ on a
\emph{Stone space} $X$ (=compact zero-dimensional space) is
\emph{automorphizable} in the sense of \cite{Me-F} (see Fact
\ref{p:aut}), that is, $X$ is a $G$-subspace of a
$G$-\emph{group} $Y$. This contrasts the case of general compact
$G$-spaces (see \cite{Me-F}).
More generally, we study 
(Theorem \ref{t:AE}) also metric and uniform versions of
automorphizable actions. As a corollary we obtain that every
ultra-metric $G$-space is isometric to a closed $G$-subset of an
ultra-normed Boolean $G$-group. This result can be treated as a non-archimedean
(equivariant) version of the classical Arens-Eells isometric linearization theorem \cite{AE}.

 \subsection{Epimorphisms in topological groups}
 \label{s:ItroEpi}

  A morphism $f: M \to G$
 in a category $\mathcal{C}$ is an {\it epimorphism} if 
 there exists no pair of distinct $g,h: G\to P$ in $\mathcal{C}$ such that
$gf=hf.$
 In the category of Hausdorff topological
 groups a morphism with a dense range is obviously an epimorphism.
 K.H. Hofmann asked in the late 1960's whether the converse is true.
 This \emph{epimorphism problem} was answered by Uspenskij \cite{Us-epic} in the negative.
Nevertheless, in many natural cases, indeed, the epimorphism $M \to
G$ must be dense. For example,
in  case that the co-domain $G$ is either locally compact or
balanced,
that is, having the coinciding left and right uniformities
(see \cite{Num}).
Using a criterion of Pestov \cite{Pest-epic} and the uniform
automorphizability of certain actions by non-archimedean groups
(see Theorem \ref{thm:auna}) we prove in Theorem \ref{epic-NA}
that any epimorphism in the category $\mathbf{NA}$ must be dense.
 Moreover, we show that if a proper closed subgroup $H$ in a Hausdorff topological group $G$ induces a non-archimedean
 uniformity $\mathcal{U}$ on $G/H$ such that $(G/H,{\mathcal U}) \in \Unif^G$, then the inclusion is not an epimorphism in the category $\mathbf{TGR}$.
 We also improve the following result of T.H. Fay \cite{Fay}: for a topological group $G$
the inclusion of a proper open subgroup $H$ is not an epimorphism.

\subsection{Graev type ultra-metrics}  In his classical work \cite{GRA}, Graev proved that every metric on $X\cup \{e\}$ admits
an extension to a maximal invariant metric on $F(X).$ In the present
work we explore (especially see Theorem \ref{t:AE}) Graev type
ultra-metrics and ultra-norms on free Boolean groups which appeared
in our previous work \cite{MS1arx}).

Graev type ultra-metrics play a major role in several recent papers.
In Section \ref{sec:gra} we briefly compare two seemingly different constructions: one of
Savchenko-Zarichnyi \cite{SZ} and the other of Gao \cite{GAO}.

\subsection{Preliminaries and notations}
\label{s:prel}

All topological groups and spaces in this
paper are assumed to be Hausdorff unless otherwise is stated (for
example, in 
Section \ref{s:final}).  The cardinality of a set $X$ is denoted by
$|X|$. All cardinal invariants are assumed to be infinite. As usual
for a topological space $X$ by $w(X), d(X), \chi(X), l(X), c(X)$ we
denote the weight, density, character, Lindel\"{o}f degree and the
cellularity, respectively. By $N_x(X)$ or $N_x$ we mean the set of
all neighborhoods at $x$.

For every
group $G$ we denote the identity element by $e$
(or by $0$ for additive groups).
 A \emph{Boolean group} is a group in which every nonidentity element is of order
 two. A topological space $X$ with a
continuous group action $\pi: G \times X \to X$ of a topological
group $G$ is called a \emph{$G$-space}.  If, in addition, $X$ is a
topological group and all $g$-translations, $\pi^g: X \to X, \ x
\mapsto gx:=\pi(g,x)$, are automorphisms of $X$ then $X$ becomes a
\emph{$G$-group}. 

We say that a topological group $G$ is \emph{complete} if it is
complete in its two-sided uniformity. For every set $X$ denote by
$F(X), A(X)$ and $B(X)$ the free group, the free abelian group and
the free Boolean group over $X$ respectively. We reserve the
notation $F(X)$ also for  the free topological  group in the sense
of Markov.

\vskip 0.2cm

\noindent \textbf{Acknowledgment:} We thank  D. Dikranjan, M. Jibladze, D. Pataraya and L. Polev for several suggestions.

\section{Some facts about non-archimedean groups and uniformities}

We mostly use the standard definition of a {\it uniform space} $(X,{\mathcal
U})$ by \emph{entourages} (see for example, \cite{Eng}). An
equivalent approach via \emph{coverings} can be found in
\cite{isbo}. We denote the induced topology by $top({\mathcal U})$
and require it to be Hausdorff, namely, $\cap\{\eps \in {\mathcal
U}\}=\bigtriangleup$. By $\textbf{Unif}$ we denote the category of
all uniform spaces.

The subset $\{\eps(a):  \eps \in {\mathcal U}\}$ is a neighborhood
base at $a \in X$ in the topological space $(X,top({\mathcal U}))$,
where $\eps(a)=\{x \in X: \ (a,x) \in \eps \}$. For a nonempty
subset $A \subset X$ denote $\eps(A)=\cup \{\eps(a): \ a \in A\}$.
We say that a subset $A \subset X$ is \emph{$\eps$-dense} if
$\eps(A)=X$.

A subfamily $\a \subset {\mathcal U}$ such that each $\eps \in {\mathcal U}$ has a refinement $\delta \subset \eps$ with
 $\delta \in \a$ is said to be a (uniform) \emph{base} of ${\mathcal U}$.
 The minimal cardinality of a base of ${\mathcal U}$ is called the \emph{weight} of ${\mathcal U}$.
 Notation: $w({\mathcal U})$. Recall that ${\mathcal U}$ is metrizable (that is, ${\mathcal U}$ is induced by a metric on $X$)
 if and only if the weight is countable, $w({\mathcal U})=\aleph_0$.

 As usual, $(X,{\mathcal U})$ is \emph{precompact} (or, \emph{totally bounded}) if for every $\eps \in {\mathcal U}$ there exists
a finite $\eps$-dense subset. By the \emph{uniform Lindel\"{o}f degree} of $(X,{\mathcal U})$ we mean the minimal (infinite) cardinal $\kappa$
 such that for each entourage $\eps \in {\mathcal U}$ there exists an $\eps$-dense subset $A_{\eps} \subset X$ of cardinality $|A_{\eps}| \leq \kappa$.
 Notation: $l({\mathcal U})=\kappa$.
 We write $(X,\mathcal{U}) \in \mathbf{Unif}(\cdot, \kappa)$ whenever $l(\mathcal{U}) \leq \kappa$.

 In  terms of coverings, $l({\mathcal U}) \leq \kappa$ means that
 every uniform covering $\eps \in \mathcal{U}$ has a subcovering $\delta$ (equivalently, a subcovering $\delta \in
 \mathcal{U}$)
 of cardinality $|\delta| \leq \kappa$.
So,  we may always choose a base $\a$ of ${\mathcal U}$ such that
$|\a| = w({\mathcal U})$ and $|\delta| \leq l({\mathcal U})$ for every $\delta \in \a$.
  Note that $l({\mathcal U}) \leq \min \{l(X),d(X),c(X)\}$ and $l(\mathcal{U}) \leq w(X) \leq w(\mathcal{U}) \cdot l(\mathcal{U})$,
 where $X=(X,top({\mathcal U}))$ is a topological space induced by ${\mathcal U}$.

 \begin{remark} \label{r:St}
The class $\mathbf{Unif}(\cdot, \kappa)$ is closed under arbitrary products, subspaces and
uniformly continuous images. $(X,\mathcal{U}) \in \mathbf{Unif}(\cdot, \kappa)$ if and only if
$(X,\mathcal{U})$ can be embedded into a product $\prod_i X_i$ of metrizable uniform spaces $(X_i, \mathcal{U}_i)$
with $w(X_i)=l(\mathcal{U}_i) \leq \kappa$ such that $|I| \leq w(\mathcal{U})$.
\end{remark}

Note that $w(\mathcal{U}) =\lambda$, $l(\mathcal{U}) = \kappa$ exactly means that the (uniform) \emph{double weight}, in the sense of \cite{Kulpa}
is $dw(\mathcal{U})=(\lambda,\kappa)$.
Denote by $\mathbf{Unif}(\lambda,\kappa)$ the class of all uniform spaces with double weight $dw(X,\mathcal{U}) \leq (\lambda,\kappa)$.

\begin{defi} \label{d:univ}
Let $\K \subset \mathbf{Unif}$ be a class of 
uniform spaces. Let us say that a uniform space $X$ is:

(a) \emph{universal} in $\K$ if $X \in \K$ and $\forall$ \ $Y \in \K$ there exists a uniform embedding $Y \hookrightarrow X$.

(b) \emph{co-universal} in $\K$ if $X \in \K$ and for every $Y \in \K$
there exists a uniformly continuous onto map $f: X \to Y$ which is a quotient map of topological spaces.
\end{defi}

In \cite{Kulpa} Kulpa proves that there exists a universal uniform
space with dimension $\leq n$ and $dw(\mathcal{U}) \leq
(\lambda,\kappa)$. Every isometrically universal separable metric
space (say, $C[0,1]$, or the \emph{Urysohn space $\U$}) provides an
example of a universal uniform space in the class
$\mathbf{Unif}(\aleph_0,\aleph_0)$. For some results about
isometrically universal spaces see \cite{Katetov,Us-Ur,Husek}.
However, seemingly it is an open question if there exists a
universal uniform space in $\mathbf{Unif}(\lambda,\kappa)$. In fact,
it is enough to solve this question for
$\mathbf{Unif}(\aleph_0,\kappa)$. Indeed, if $(X,{\mathcal U})$ is
universal in
 $\mathbf{Unif}(\aleph_0,\kappa)$ then the uniform space $X^{\lambda}$ is universal in $\mathbf{Unif}(\lambda,\kappa)$.
 In order to see this recall that $\mathbf{Unif}(\cdot, \kappa)$ is closed under products, subspaces and uniformly continuous images (Remark \ref{r:St}).

For a topological group $G$ we have four natural uniformities:
\emph{left}, \emph{right}, \emph{two-sided} and \emph{lower}.
Notation: ${\mathcal U}_l, {\mathcal U}_r, {\mathcal
U}_{l \vee r}$, ${\mathcal U}_{l \wedge r}$ respectively. Note that
the weight of all these uniformities is equal to $\chi(G)$, the
\emph{topological character} of $G$. Also, $l({\mathcal
U}_l)=l({\mathcal U}_r)=l({\mathcal U}_{l \vee r})$. This invariant,
the \emph{uniform Lindel\"{o}f degree} of $G,$ is denoted by
$l^u(G)$. Always, $l({\mathcal U}_{l \wedge r}) \leq l^u(G)$.
Note that $l^u(G) \leq \kappa$ if and only if
$G$ is $\kappa$-\emph{bounded} in the sense of Guran
($\kappa$-\emph{narrow} in other terminology). That is, for every $U
\in N_e(G)$ there exists a subset $S \subset G$ such that $US=G$ and
$|S| \leq \kappa$.

\begin{lemma} \label{l:UnLind}
Let $G$ be a topological group.
\ben
\item $w(G)=\chi(G) \cdot l^u(G)$ and $dw(G,{\mathcal U}_{l \vee r}) = (\chi(G),l^u(G)) \leq (w(G),w(G))$.
\item
(Guran, see for example \cite[Theorem 5.1.10]{AT}) $l^u(G) \leq \kappa$ if and only if 
$G$ can be embedded into a product $\prod_i G_i$ of topological groups $G_i$
of topological weight $w(G_i) \leq \kappa$.
    \item \cite{DTk} (see also \cite[p. 292]{AT}) Let 
    $X \subset G$ topologically generate $G$. Consider the induced uniform subspace $(X,{\mathcal U}_X)$ where ${\mathcal U}_X={\mathcal U}_{l \vee r}|X$. Then $l^u(G) = l({\mathcal U}_X)$.
\een
\end{lemma}

Every uniform space $(X,{\mathcal U})$ is uniformly embedded into an (abelian) topological group $G$ such that
$w({\mathcal U})=\chi(G)$ and $l({\mathcal U}) = l^u(G)$
(i.e., $dw({\mathcal U})=dw(G,{\mathcal U}_{l \vee r})$).
In order to see this one may use Arens-Eells embedding theorem \cite{AE} taking into account Lemma \ref{l:UnLind}.3.

\subsection{Non-archimedean uniformities}

Monna (see \cite[p.38]{ro} for
more details) introduced the notion of {\it non-archimedean} uniform spaces.
A uniform space $X$ is non-archimedean if it has a base $B$ consisting of equivalence relations
 (or, partitions, in the language of coverings)
on $X.$  It is also equivalent to say that for such a space the
\emph{large uniform dimension} (in the sense of \cite[p. 78]{isbo})
is zero. For a uniform space $(X,\mathcal{U})$ denote by $Eq({\mathcal U})$ the set of all
equivalence relations on $X$ which belong to ${\mathcal U}$.

Recall that every compact space has a unique compatible uniformity.
A \emph{Stone space} is  a compact zero-dimensional space.  It is
easy to see that such a space is always non-archimedean. 
  There exist $2^{\aleph_0}$-many nonhomeomorphic metrizable Stone spaces.

A metric space $(X,d)$ is  an \emph{ultra-metric space} (or,
\emph{isosceles} \cite{Lem03})
if $d$ is an \emph{ultra-metric}, i.e., it satisfies {\it the strong
triangle inequality}
$$d(x,z) \leq max \{d(x,y), d(y,z)\}.$$

Allowing the distance between distinct elements to be zero we obtain
the definition of  an \emph{ultra-pseudometric}.
For every ultra-pseudometric $d$ on $X$ 
the open balls of radius $\varepsilon >0$ form a clopen partition of
$X.$ So, the uniformity induced by any ultra-pseudometric $d$ on $X$
is non-archimedean. A uniformity is non-archimedean if and only if
it is generated by a system $\{d_i\}_{i \in I}$ of {\it
ultra-pseudometrics}.
%


Let us say that a uniformity ${\mathcal U}$ on $X$ is \emph{discrete} if
${\mathcal U}=P(X \times X)$ (or, equivalently, $\Delta:=\{(x,x): x \in X\} \in
{\mathcal U}$).

Denote by $\kappa^{\lambda}$ the power space of the discrete uniform space with cardinality $\kappa$.

\begin{lemma} \label{l:UnLindNA}
\ben
\item The Baire space $B(\kappa)=\kappa^{\aleph_0}$ is a universal uniform space in the class $\mathbf{Unif}_{\sna}(\aleph_0,\kappa)$ of all metrizable
non-archimedean uniform spaces $(X,{\mathcal U})$ such that $l({\mathcal U}) \leq \kappa$.
    \item The generalized Baire space $\kappa^{\lambda}$ is a universal uniform space in $\mathbf{Unif}_{\sna}(\lambda,\kappa)$.
\een

\end{lemma}

Let $(X,{\mathcal U})$ be a non-archimedean uniformity. By a result
of R. Ellis \cite{Ellis} for every uniform ultra-pseudometric on a
subset $Y \subset X$ there exists an extension to a uniform
ultra-pseudometric on $X$. Another result from \cite{Ellis} shows
that, in fact, $B(\kappa)$ is also co-universal in the class
$\mathbf{Unif}_{\sna}(\aleph_0,\kappa)$  (see  Section \ref{s:co-un}
below).

\subsection{Non-archimedean groups}

Recall that a topological group is said to be {\it non-archimedean}
if it has a local base at the identity consisting of open subgroups.
All $\mathbf{NA}$ groups are totally disconnected. The converse is
not true in general (e.g., the group $\Q$ of all rationals). However, in case that a totally disconnected
group $G$
is also locally compact then both $G$ and $Aut(G),$ 
the group of all automorphisms of $G$ endowed with the
\emph{Birkhoff topology}, are $\mathbf{NA}$ (see Theorems 7.7 and
26.8 in \cite{HR}).

The prodiscrete groups (= inverse limits of discrete groups) are in
$\mathbf{NA}.$ Every complete balanced $\mathbf{NA}$ group (in particular, every profinite group) is
prodiscrete.

\begin{example}\label{exa:nam}
We list here some non-archimedean groups: \ben
 \item $(\Z, \tau_p)$ where $\tau_p$ is the $p$-adic topology on the set of all integers $\Z$.
 \item The symmetric topological group $S_X$ with the topology of
 pointwise convergence. Note that $S_X$ is not balanced for any infinite set $X$.
 \item   $\Homeo(\{0,1\}^{\aleph_0})$, the homeomorphism group of the Cantor cube, equipped with the
 compact-open topology.
 \een
\end{example}

The $\mathbf{NA}$ topological groups from Example \ref{exa:nam} are
all \emph{minimal}, that is, each of them does not admit a strictly
coarser Hausdorff group topology. By a result of Becker-Kechris
\cite{bk} every second countable (Polish) $\mathbf{NA}$ group is
topologically isomorphic to a (resp., closed) subgroup of the
symmetric group $S_{\N}$. So, $S_{\N}$ is a universal group in the
class of all second countable $\mathbf{NA}$ groups. In fact, a more
general result remains true: $S_X$ is a universal group in the class
of all $\mathbf{NA}$ groups $G$ with the topological weight $w(G)
\leq |X|$, where $|X|$ is the cardinality of the infinite set $X$.
See for example, \cite{Hig,MS1} and also Fact \ref{t:condit} below.

By results of \cite{MS1} there are many minimal $\mathbf{NA}$
groups: every $\mathbf{NA}$ group is a group retract of a minimal $\mathbf{NA}$ group.
See Section \ref{s:H} below and also survey papers on minimal groups \cite{D-s,DM-s}.

Teleman \cite{Te} proved that every topological group is a subgroup of $\Homeo(X)$ for some compact 
$X$  and, it is also a subgroup of $Is(M,d)$, the topological group
of isometries of some metric space $(M,d)$ equipped with the
pointwise topology (see also \cite{pest-wh}).
Replacing  "compact" with "compact
zero-dimensional" and "metric"
with "ultra-metric" we obtain  characterizations for the class $\mathbf{NA}$ (see \cite{Les86} and Fact \ref{t:condit} below). 



The class  $\mathbf{NA}$ is a {\it variety} in the sense of \cite{Mor},
i.e., it is closed under taking subgroups, quotients
and arbitrary products. Furthermore,
$\mathbf{NA}$ is closed under group extensions  
(see \cite[Theorem 2.7]{Hig}). In particular, $\mathbf{NA}$ is
stable under semidirect products. We collect here some
characterizations of non-archimedean groups, majority of which  are
known. For details and more results see \cite{Les86,Pest-free, MS1}.

\begin{fact} \cite{MS1} \label{t:condit}
The following assertions are equivalent: \ben
\item $G$ is a non-archimedean topological group.

\item The right (left, two-sided, lower) uniformity on $G$ is non-archimedean.
\item $\dim \beta_G G =0$, where $\beta_G G$ is the \emph{maximal $G$-compactification} \cite{MES01} of $G$.
\item $G$ is a topological subgroup  of   $\Homeo(X)$ for some Stone
space $X$ (where $w(X)=w(G)$).
\item $G$ is a topological subgroup of the automorphism group (with the pointwise topology) $Aut(V)$ for some
discrete 
Boolean ring $V$ (where $|V|=w(G)$).
\item $G$ is a topological
subgroup of the group $Is_{Aut}(M)$ of all norm preserving
automorphisms of some ultra-normed Boolean group $(M,||\cdot||)$
(where $w(M)=w(G)$).
\item
$G$ is embedded into the symmetric topological group $S_{\kappa}$
(where $\kappa=w(G)$).
\item
$G$ is a topological subgroup of the group $Is(X,d)$ of all
isometries of an 
ultra-metric space $(X,d)$, with the topology of pointwise
convergence
(where $w(X)=w(G)$). 
\item The right (left) uniformity on $G$ can be generated by a
system $\{d_i\})_{i \in I}$ of right (left) invariant 
ultra-pseudometrics of cardinality $|I| \leq \chi(G)$.

\item $G$ is a topological subgroup of the automorphism group $Aut(K)$ for some compact abelian group $K$ (with $w(K)=w(G)$).
\item $G$ is a topological subgroup of the automorphism group $Aut(K)$ for some profinite group $K$ (with $w(K)=w(G)$).
\een
\end{fact}


\vskip 0.3cm

An \emph{ultra-seminorm} on a topological group $G$ is a 
function $p: G \to \R$ such that
\ben
\item $p(e)=0$;
\item $p(x^{-1}) =p(x)$;
\item $p(xy) \leq max\{p(x),p(y)\}$.
\een

Always, $p(x) \geq 0$. We call $p$  an
\emph{ultra-norm} if in addition $p(x)=0$ implies $x=e$.
For ultra-seminorms on an abelian additive group $(G,+)$ we prefer
the notation $||\cdot ||$ rather than $p$.
For every
ultra-seminorm on $G$ and every $a \in G$ the function
$q(x):=p(axa^{-1})$ is also an ultra-seminorm on $G$. We say that
$p$ is \emph{invariant} if $p(axa^{-1})=p(x)$ for every $a,x \in G$.
 We say that a pseudometric $d$ on $G$ is \emph{invariant} if it is left and right invariant.

\begin{lemma} \label{l:ultraProp}
\ben
\item
For every ultra-seminorm $p$ on $G$ we have: \ben
 \item  $H_{\eps}:=\{g \in G: \ p(g)< \eps \}$
is an open subgroup of $G$ for every $\eps>0$.
\item
The function $d: G \times G \to \R$ defined by $d(x,y):=p(x^{-1}y)$
is a left invariant ultra-pseudometric on $G$ and $p(x)=d(e,x)$.
\een If $p$ is invariant then $H_{\eps}$ is a normal subgroup in $G$
and $d$ is invariant.

\item Let $G \times X \to X$ be an action of a group $G$ on a set $X$.
If $d$ is a $G$-invariant ultra-pseudometric on  $X$ and $x_0 \in X$ is
a point in $X$ then $p(g):=d(x_0,gx_0)$ is an ultra-seminorm on $G$.
    \item 
   As a particular case of $(2)$,  for every left invariant ultra-pseudometric $d$ on $G$ we have the  ultra-seminorm
$
p(x):=d(e,x).
$
Here $p$ is invariant if and only if $d$ is invariant.

\item For every topological group $G$ and an open subgroup $H$ of $G$ there exists a continuous ultra-seminorm $p$ on $G$ such that $\{x \in G: \ p(x) < 1 \} = H$.
If, in addition, $H$ is normal in $G$ then we can assume that $p$ is
invariant.
\item A homomorphism $f: G \to H$ from a topological group $G$ into a non-archimedean group $H$ is continuous
if and only if for every continuous ultra-seminorm $p$ on $H$ the 
ultra-seminorm $q: G \to \R$ defined by $q(g):=p(f(g))$ is continuous.
\een
\end{lemma}
\begin{proof}
(1) (a)
Indeed, if $p(x) < \eps$ and $p(y) < \eps$ then
$$p(xy^{-1}) \leq \max \{p(x),p(y^{-1})\}=\max \{p(x),p(y)\} < \eps.$$

If $p$ is invariant then $H_{\eps}$ is normal in $G$ since $p(axa^{-1})=p(x) < \eps$ for every $a \in G$.

 (b) of (1) is trivial and
(2), (3) and (5) are straightforward.

(4) Define the  ultra-seminorm on $G$ as $p(g)=0$ for $g \in
H$ and $p(g):=1$ if $g \notin H$.
\end{proof}

\vskip 0.3cm

A topological group $G$ is {\it balanced} (or, $\mathbf{SIN}$) if its left and right uniform
structures coincide (see for example \cite{RD}).
It is equivalent to say that $G$ has small neighborhoods which are invariant under conjugations.
That is, for every $U \in N_e(G)$ there exists $V \in N_e(G)$ such that $gVg^{-1}=V$ for every $g \in G$.
Furthermore, $G$ is balanced if and only if the uniformity on $G$
can be generated by a system of invariant pseudometrics (or
invariant seminorms).


\begin{lemma} \label{l:bal} For a balanced group $G$ the following
conditions are equivalent:
\ben
\item $G \in \mathbf{NA}$; 
\item $G$ has a local base at the identity consisting of open normal subgroups;
\item  $G$ is embedded into a product $\prod_{i \in I} G_i$ of discrete groups, where $|I| \leq \chi(G)$;
\item the 
uniformity on $G$ can be generated by a system $\{d_i\}_{i \in I}$ of invariant ultra-pseudometrics (ultra-seminorms),
where $|I| \leq \chi(G)$.
\een
\end{lemma}
\begin{proof}
(1) $\Rightarrow$ (2): Let $V$ be a neighborhood of $e$ in $G$. We
have to show that there exists an open normal subgroup $M$ of $G$
such that $M \subseteq V$. Since $G  \in \mathbf{NA}$, 
there exists
an open subgroup $H$ of $G$ such that $H \subseteq V$. Since $G$ is
balanced $N:=\cap_{g \in G} gHg^{-1}$ is again a \nbd  \ of $e$.
Then $N$ is a normal subgroup of $G$ and $N \subseteq V$. Clearly
the subgroup $N$ is open because its interior is nonempty.

(2) $\Rightarrow$ (3): For every open, and hence closed, normal subgroup $N$ of $G$ the corresponding factor-group $G/N$ is discrete.

(3) $\Rightarrow$ (4): For every discrete group $P$ the usual $\{0,1\}$-ultra-metric is invariant.


(4) $\Rightarrow$ (1): Let $p$ be an invariant ultra-seminorm on $G$. Then the set
$$H_{\eps}:=\{g \in G: \ p(g) < \eps \}$$ is an open (normal) subgroup of $G$ by Lemma  \ref{l:ultraProp}.1.
\end{proof}

Every complete balanced $\mathbf{NA}$ group $G$ is a prodiscrete group
as it follows from assertion (3) and standard properties of projective limits (see e.g., \cite[Prop. 2.5.6]{Eng}).



\begin{lemma} \label{l:MetrBal}
Let $G \in \mathbf{NA}$.
\ben
\item
$G$ is metrizable iff its right (left)
uniformity can be generated by a single right (left) invariant ultra-metric $d$ on $G$.
\item
$G$ is metrizable and balanced iff its right (left)
uniformity can be generated by a single 
invariant ultra-metric $d$ on $G$.
\een
\end{lemma}
\begin{proof} (1):
If $G$ is metrizable then the right  (left) uniformity of $G \in \mathbf{NA}$ can be generated by a countable
system $\{d_n\}_{n \in \N}$ of right (left) invariant ultra-pseudometrics (cf. Fact \ref{t:condit}.9).
One may assume in addition that $d_n \leq 1$. Then
the desired right (left) invariant ultra-metric on $G$ can be defined by $d(x,y):=\sup_{n \in \N} \{\frac{1}{2^n}d_n(x,y)\}$.

(2):
If $G$ is metrizable and balanced then one may assume in the proof of (1) that each $d_n$ is invariant (see Lemma \ref{l:bal}). Therefore, $d$ is also invariant.

Conversely, if the right (left) uniformity of $G$ can be generated by an invariant ultra-metric $d$ then clearly,
$G$ is metrizable and balanced.
\end{proof}



\section{Uniform free $\mathbf{NA}$ topological groups} \label{sec:ufna}

By 
$\mathbf{TGr}$ we denote the category of all 
topological
groups. By 
$\mathbf{AbGr}$, $\textbf{Prec}$, $\textbf{Pro}$ we denote its
full subcategories of all 
 abelian, precompact, and
\emph{profinite} (= inverse limits of finite groups) groups respectively. Usually we denote a category and its
class of all objects by the same symbol.

In this section, unless 
otherwise is stated,  all topological groups are considered with
respect to  the two sided uniformity ${\mathcal U}_{l \vee r}$.
Assigning to every 
topological group $G$ the uniform space $(G,{\mathcal U}_{l \vee
r})$ defines a forgetful functor
from the category of all 
topological groups $\mathbf{TGr}$ 
 to the category of
all 
uniform spaces 
$\mathbf{Unif}.$

\begin{defi} \label{d:FreeGr} Let $\Omega$ be a subclass of $\mathbf{TGr}$ 
and 
 $(X,{\mathcal U}) \in \mathbf{Unif}$ be a uniform space. By an
\emph{$\Omega$-free topological group of $(X,{\mathcal U})$} we
mean a pair $(F_{\Omega}(X,{\mathcal U}),i)$ (or, simply,
$F_{\Omega}(X,{\mathcal U})$, when $i$ is understood), where
$F_{\Omega}(X,{\mathcal U})$ is a topological group from $\Omega$
and $i: X \to F_{\Omega}(X,{\mathcal U})$ is a uniform map
satisfying the following universal property. For every uniformly
continuous map \textbf{$\varphi: (X,{\mathcal U}) \to G$} into a
topological group $G \in \Omega$ there exists a unique 
continuous homomorphism \textbf{$\Phi: F_{\Omega}(X,{\mathcal U})
\to G $} for which the following diagram commutes:
\begin{equation*} \label{equ:ufn}
\xymatrix { (X,{\mathcal U}) \ar[dr]_{\varphi} \ar[r]^{i} & F_{\Omega}(X,{\mathcal U})
\ar[d]^{\Phi} \\
  & G }
\end{equation*}
\end{defi}

If $\Omega$ is a 
subcategory of $\mathbf{TGr}$ then a categorical reformulation of this definition is that
$i: X \to F_{\Omega}(X,{\mathcal U})$ is a universal arrow from $(X,{\mathcal U})$ to the forgetful functor
$\Omega \to \mathbf{Unif}$.

\begin{remark} \label{r:short}
Also we use a shorter notation dropping ${\mathcal U}$ (and $X$) when the uniformity (and the space) is understood.
For example, we may write $F_{\Omega}(X)$ (or, $F_{\Omega}$) instead of $F_{\Omega}(X,{\mathcal U})$.
\end{remark}

Every Tychonoff space $X$  admits the greatest compatible uniformity, the so-called  \emph{fine uniformity}, which we
denote by $\mathcal{U}_{max}$. The corresponding free group $F_{\Omega}(X,{\mathcal U_{max}})$ is denoted by
$F_{\Omega}(X)$ and is called the \emph{$\Omega$-free topological group} of $X$. For $\Omega=\mathbf{TGr}$ and $\Omega=\mathbf{AbTGr}$ we get the classical \emph{free topological group} and \emph{free abelian topological group} (in the sense of Markov) of $X$
keeping the standard notation: $F(X)$ and $A(X)$.


\subsection{The existence}

\begin{defi}\label{def:semiv}
A nonempty subclass $\Omega$ of 
 $\mathbf{TGr}$ is said to be:
\ben
\item \emph{$\overline{S}C$-variety} (see \cite{AT}) if
$\Omega$ is closed under:
a) cartesian products;
b) \textbf{closed} subgroups.
\item \emph{$SC$-variety} if
$\Omega$ is closed under:
a) cartesian products;
b) subgroups.
\item \emph{variety} (see \cite{Mor})
 if $\Omega$ is closed under:
 a) cartesian products;
 b) subgroups;
 c) quotients.
\een
\end{defi}

Note that while \textbf{Pro} is an $\overline{S}C$-variety, all other subclasses $\Omega$ of 
$\mathbf{TGr}$ from Remark \ref{r:diag} are varieties.

\begin{thm} \label{thm:Samuel}
Let $\Omega$ be a subclass of 
 $\mathbf{TGr}$ which is an $\overline{S}C$-variety and
$(X,{\mathcal U})$ be a uniform space.
 \ben
\item The uniform free topological group $F_{\Omega}:=F_{\Omega}(X,{\mathcal U})$ exists.

\item $F_{\Omega}$ is unique up to a topological group isomorphism.

\item
\ben
\item
 $F_{\Omega}$ is topologically  generated by $i(X) \subset F_{\Omega}$.
\item
 For every uniform map $\varphi: (X,{\mathcal U}) \to G$ into a topological group $G \in \Omega$
there exists a continuous homomorphism $\Phi: F_{\Omega} \to G $ such that $\Phi \circ i=\varphi$.
\een

 Moreover, these two properties characterize $F_{\Omega}(X,{\mathcal
U})$.
\item If $\Omega$ is an $SC$-variety then $F_{\Omega}$ is algebraically  generated by $i(X)$.
\een
\end{thm}
\begin{proof}
$(1):$ Existence.
We give here a standard categorical construction (with some  minor adaptations)
which
goes back to Samuel and Kakutani.
Denote 
 $\mathfrak{m}:=\max(|X|,\aleph_0)$.
Let
$\mathfrak{F}$ be a subclass of $\Omega$
 such that
$|G|\leq 2^{2^{\mathfrak{m}}}$ for $G\in \mathfrak{F},$ distinct
members of $\mathfrak{F}$ are not topologically isomorphic, and
every topological group $H$ for which $|H|\leq 2^{2^{\mathfrak{m}}}$
is topologically isomorphic with some $G\in \mathfrak{F}.$ Let
$\{(G_j,\varphi_j)\}_{j\in J}$ 
 consist of all pairs
$(G_j,\varphi_j)$ where $G_j\in \mathfrak{F}$ and $\varphi_j$ is a
uniformly continuous mapping of $X$ into $G_j.$ It is easy to see
that $\mathfrak{F}$ is a set. Then $J$ is a set as well. If $H \in
\Omega$ is a topological group, $|H|\leq 2^{2^{\mathfrak{m}}},$ and
$\varphi$ is a uniformly continuous  mapping of $X$ into $H$, then
there is  $j_0\in J$ and a topological isomorphism $\tau:G_{j_0}\to
H$ such that  $\tau\circ \varphi_{j_0}=\varphi.$ In such a case we
identify the pair $(H,\phi)$ with the pair
$(G_{j_0},\varphi_{j_0}).$  Let $M=\prod_{j\in J}G_j.$  For $x\in
X,$ define $i(x)\in M$  by $i(x)_j=\varphi_j(x).$ Finally,
let $F_{\Omega}:=F_{\Omega}(X,{\mathcal U})$ be the closed subgroup of $M$
topologically generated by $i(X).$ Since the class $\Omega$ is an
$\overline{S}C$-variety, both $M$ and $F_{\Omega}$
are in $\Omega$ by  conditions (a) and (b) of Definition
\ref{def:semiv}.1. Clearly,
 $i:(X,{\mathcal U})\to F_{\Omega}$ is uniformly continuous. Now, if
$\varphi$ is a uniformly continuous mapping of $X$ into any
 topological group $G \in \Omega,$  the image $\varphi(X)$ in $G$ is contained
in the subgroup $P:=cl(<\varphi(X)>)$ of $G$, where $<\varphi(X)>$ is the subgroup of $G$ algebraically generated by $\varphi(X)$.
Since $|<\varphi(X)>| \leq \mathfrak{m}=\max \{|X|, \aleph_0\}$ and $P$ is Hausdorff
we have $|P| \leq 2^{2^{\mathfrak{m}}}.$
Thus, by our assumption on $\mathfrak{F},$ the pair $(P,\varphi)$ is
isomorphic to a pair $(G_{j_0},\varphi_{j_0})$ for some $j_0 \in J$.
Let $\pi_{j_0}: F_{\Omega} \to G_{j_0}$ be the
restriction on $F_{\Omega} \subseteq \prod_{j \in J}G_j$ of the projection onto the $j_0$-th axis. Then $\varphi= \Phi
\circ i$, where
 $\Phi:=\tau \circ \pi_{j_0}.$ Finally, note that $\Phi$ is unique since
$<i(X)>$ is a dense subgroup of  $F_{\Omega}$ and $G$ is Hausdorff.

$(2):$ Uniqueness. Assume that there exist Hausdorff topological
groups $F_1,F_2$ and uniformly continuous maps $i:X\to F_1, \ j:X\to
F_2$ such that  
the pairs $(i,F_1)$ and $(j,F_2)$ satisfy
the universal property. Then by Definition \ref{d:FreeGr} there exist \emph{unique} continuous
homomorphisms $\Phi_1:F_2\to F_1, \Phi_2:F_1\to F_2$ such that
$\Phi_2\circ i=j, \ \Phi_1\circ j=i.$
For $\Phi\in\{\Phi_1\circ
\Phi_2,Id_{F_1}\}$ we have $\Phi\circ  i=i,$ and thus $\Phi_1\circ
\Phi_2=Id_{F_1}.$ Similarly, 
$\Phi_2\circ \Phi_1=Id_{F_2}.$ Therefore, $\Phi_2:F_1\to F_2$ is a topological
group isomorphism.
%

(3) Assertion $(a)$ follows from the constructive description of
$F_{\Omega}$ given in the proof of $(1),$ and from
$(2).$ Property $(b)$ is a part of the definition of
$F_{\Omega}.$ These two properties characterize
$F_{\Omega}$ since
 the latter group is Hausdorff.

(4) As  an $SC$-variety $\Omega$ is closed under (not necessarily
closed) subgroups. So in the constructive description appearing in
the proof of $(1)$ we may define $F_{\Omega}$ as the
subgroup \emph{algebraically} generated by $i(X).$ Apply $(2)$ to conclude
the proof.
\end{proof}

The completion of $G$ with respect to the two-sided uniformity is denoted by $\widehat{G}$.
The proof of the following observation is straightforward.

\begin{lemma} \label{l:compl}
Let $\Omega$ be an $\overline{S}C$-variety. Denote by
$\Omega_C$ its subclass of all complete groups from $\Omega$. Then $F_{\Omega_C}=\widehat{F_{\Omega}}$.
\end{lemma}

\subsection{Classical constructions}

For $\Omega=\mathbf{TGr}$ the universal object
$F_{\Omega}(X,{\mathcal U})$  is the \emph{uniform free topological
group} of $(X,{\mathcal U}).$ Notation: $F(X,{\mathcal U}).$  This
was invented by Nakayama and studied by Numella \cite{Num2} and Pestov \cite{pes85, Pest-Cat}.

In particular, Pestov  described
 the topology of $F(X,{\mathcal U})$  (\cite{pes85}, see also  Remark \ref{rem:pests} below).
If $\Omega =\mathbf{AbGr}$ then
$F_{\Omega}(X,{\mathcal U})$ is the
 \emph{uniform free
abelian topological group} of $(X,{\mathcal U})$. Notation:
$A(X,{\mathcal U}).$ In \cite{Sip}  Sipacheva  used Pestov's
description of the free topological group $F(X)$ to generate a description of
 the free abelian topological group $A(X).$ Similarly, one can prove
 the following:

 \begin{thm} (compare with \cite[page 5779]{Sip})
Let $(X,{\mathcal U})$ be a uniform space.  For each $n\in \N$, we
fix an arbitrary entourage $W_n\in {\mathcal U}$ of the diagonal in
$X \times X$ and set $W = \{W_n\}_{n\in \N},$ $$ U(W_n) = \{\epsilon
x- \epsilon y: (x, y)\in W_n, \epsilon= \pm 1\},$$ and
$$\widetilde{U}(W) =\bigcup_ {n\in \N} (U(W_1) + U(W_2) + \cdots + U(W_n)).$$ The sets $\widetilde{U}(W)$,
where $W$ are all sequences of uniform entourages of the diagonal,
form a neighborhood base at zero for the topology of the uniform
free abelian topological group $A(X,{\mathcal U}).$
\end{thm}

Recall a  classical result concerning the (non)metrizability of free topological groups.

\begin{thm}\cite[Theorem 7.1.20]{AT}\label{thm:tkar}
If a Tychonoff space $X$ is non-discrete, then neither $F(X)$ nor
$A(X)$ are metrizable. 
\end{thm}

 Theorem \ref{thm:tkar} has a uniform modification.
In fact, we can mimic the proof of Theorem \ref{thm:tkar} to obtain
the following:

\begin{thm} \label{non-met}
Let ${\mathcal U}$ be  a non-discrete uniformity on $X$ and
$$\Omega\in \{\mathbf{TGr}, \mathbf{Prec},\mathbf{SIN},
\mathbf{AbGr}\}.$$ Then $F_{\Omega}(X,{\mathcal U})$ is not
metrizable.
\end{thm}

Contrast this result with Theorem \ref{t:metrizability},
where we show that for some natural subclasses of $\Omega =\mathbf{NA}$ the free group $F_{\Omega}(X,{\mathcal U})$ is metrizable
 whenever $(X,{\mathcal U})$ is metrizable.

\subsection{Free groups in some subclasses of $\mathbf{NA}$}

In Remark \ref{r:diag} we gave a list of some classes $\Omega$ and the corresponding free groups.
We keep the corresponding notations.

\begin{thm} \label{thm:fnag} Let 
 $(X,{\mathcal U})$ be a non-archimedean uniform space
and  
$$G\in \{ F_{\sna},F^b_{\sna}, A_{\sna},B_{\sna}\}.$$ Then:
 \ben
\item   The universal
morphism $i:(X,{\mathcal U}) \to G $ is a uniform embedding.
\item If  $G\in
\{F_{\sna},F^b_{\sna}\}$ then $G$ is 
algebraically free over $i(X).$ \\ So, if $G=A_{\sna}$ or $G=B_{\sna}$
then $G$ is algebraically isomorphic to $A(X)$, or $B(X)$, respectively.
\item  $i(X)$ is a closed subspace of $G$. 
\een
\end{thm}
\begin{proof}
 $(1):$
It suffices
to prove that the universal morphism $i:X\to B_{\sna}$ is a uniform embedding.
We show the existence of a Hausdorff $\mathbf{NA}$ group topology $\tau$
on the free Boolean group $B(X)$,
and a uniform embedding $\iota:(X,{\mathcal U})\to (B(X),\tau),$
 that clearly will imply that $i:(X,{\mathcal U}) \to B_{\sna}$
is a uniform embedding.

Consider the natural set embedding $$\iota:
X \hookrightarrow B(X), \  \iota(x)=\{x\}.$$ We  identify $x \in X$
with $\iota(x)=\{x\} \in B(X).$ Let $\mathcal{B}:=\{<\eps>\}_{\eps
\in Eq({\mathcal U})}$, where  $Eq({\mathcal U})$ is the set of
equivalence relations from ${\mathcal U}$ and $<\eps>$ is the
subgroup of $B(X)$ algebraically generated by the set
$$\{x+y \in B(X): (x,y) \in \eps\}.$$ Now, $\mathcal{B}$ is a filter base on $B(X)$
and  $\forall b\in B(X) \ \forall \eps\in Eq(\mathcal U)$ we have
$$<\eps>+<\eps>=-<\eps>=b+<\eps>+b.$$
 It follows that  there exists a
$\mathbf{NA}$ group topology $\tau$ for which $\mathcal{B}$ is a
local base at the identity. To prove that this topology is indeed
Hausdorff, we have to show that if $u\neq 0$ is of the form
$u=\sum_{i=1}^{2n} a_i$ where $n\in \N$ and $a_i\in X \ \forall
1\leq i\leq 2n,$ then there exists $\eps\in Eq(\mathcal U)$ such
that $u\notin <\eps>.$ Since $(X,{\mathcal U})$ is a Hausdorff
uniform space there exists $\eps\in {\mathcal U}$ such that
$(a_i,a_j)\notin \eps $ for every $i\neq j.$ Assuming the contrary,
let $u\in <\eps>.$ Then there exists a minimal $m\in \N$ such that
$u=\sum_{i=1}^m(x_i+y_i)$ where $(x_i,y_i)\in \eps \ \forall 1\leq
i\leq m.$ Without loss of generality we may assume that there exists
$1\leq i_0\leq m$ such that $a_1=x_{i_0}.$ Note that $y_{i_0}\neq
a_j$ for every $1\leq j\leq 2n,$ since otherwise we  obtain a
contradiction to the minimality of $m$ or to the definition of
$\eps.$ Since $B(X)$ is the free Boolean group over $X$ and $\eps$
is symmetric we can assume without loss of generality that there
exists $r\neq i$ such that $y_{i_0}=x_r.$ It follows that
$(x_{i_0}+y_{i_0})+(x_r+y_r)=x_{i_0}+r.$ Since $\eps$ is transitive
we also have $(x_{i_0},y_r)\in \eps$ and we obtain a
contradiction to the minimality of $m.$ Therefore, $\tau$ is Hausdorff.\\
We show that $ \iota:(X,{\mathcal U})\to (B(X),\tau)$ is uniformly
continuous. To see this observe that if $(x,y)\in \eps$ then $x+y\in
<\eps>.$ Finally, assume that $(x,y)\in X\times X$ such that $x+y\in
<\eps>$ where $\eps$ is an equivalence relation. We show that
$(x,y)\in \eps$ and
 conclude that $ \iota:(X,{\mathcal U})\to
(B(X),\tau)$ is a uniform embedding. Since $x+y\in <\eps>$ there
exists a natural number $n$ such that $x+y=\sum_{i=1}^n(a_i+b_i),$
where $(a_i,b_i)\in \eps \ \forall 1\leq i \leq n.$ Moreover, $n$
may be chosen to be minimal. By the definition of $B(X)$ and the
fact that $\eps$ is symmetric we may assume without loss of
generality that there exists $1\leq i_0\leq n$ such that
$a_{i_0}=x.$  The case  $b_{i_0}=y$ is trivial. So we can
assume that $b_{i_0}\neq y.$  Since $B(X)$ is  the free Boolean
group over $X$ there exists $i_1\neq i_0$ such that either
$b_{i_0}=a_{i_1}$ or $b_{i_0}=b_{i_1}.$ In the former case we have
$(a_{i_0}+b_{i_0})+(a_{i_1}+b_{i_1})=a_{i_0}+b_{i_1}$ and
$(a_{i_0},b_{i_1})\in \eps,$ since $\eps$ is transitive, which
contradicts  the minimality of $n.$  The latter case
yields  a similar contradiction.\\
$(2):$ We  can now identify $X$ with $i(X).$  Denote by
$X/\varepsilon$ the quotient set equipped with the discrete
uniformity. The function $f_{\varepsilon}:X\to X/\varepsilon$ is the
uniformly continuous map which maps every $x\in X$ to the
equivalence class $[x]_{\varepsilon}.$ We first deal with the case
$G=F_{\sna}.$ We reserve the notation
$f_{\varepsilon}$ also for the homomorphic extension from
$F_{\sna}$ to  the (discrete) group
$F(X/\varepsilon).$ This allows us to define
$[w]_{\varepsilon}:=f_{\varepsilon}(w)$ for every $w\in
F_{\sna}.$ Let $w=x_1^{t_1}\cdots x_k^{t_k}\in
F_{\sna}$ where $t_i\in \Bbb Z\setminus \{0\}$ for
every $1\leq i\leq k$   and $x_i\neq x_{i+1}$ for every $ 1\leq
i\leq k-1.$ If $w$ is of the form $x^t,$  one can consider the
extension $\overline{f}: F_{\sna} \to \Z$ of the
constant function $f:(X,{\mathcal U}) \to \Z, \ f\equiv 1.$ We have
$$\overline{f}(w)=t\neq 0.$$ Otherwise, assume that $w$ is not of the form
$x^t.$ Since $(X,{\mathcal U})$ is a non-archimedean Hausdorff space
there exists an equivalence relation $\varepsilon\in {\mathcal U} $
such that
 $$(x_i,x_{i+1})\notin \varepsilon \
\forall i\in \{1,2,\ldots, k-1\}.$$
Since $F(X/\varepsilon)$  is algebraically free it follows that
$f_{\varepsilon}(w)\neq e_{F(X/\varepsilon)}.$  The groups
 $F(X/\varepsilon)$  and $\Z$ (being discrete groups) are both non-archimedean
 Hausdorff.
So we can conclude that
$F_{\sna}$ is algebraically free over $X$. \\For the case
$G=F^b_{\sna}$ use the fact that $\Z$ and $F(X/\varepsilon)$
are also balanced.\\ For the case $G=A_{\sna}$ one may use the
fact that $\Z$ is also abelian and replace $F(X/\varepsilon)$ with
$A(X/\varepsilon)$. Up to minor changes the proof is
similar to the proof of the case $G=F^b_{\sna}$.
For the Boolean case replace $\Z$ with $\Z_2$ and $F(X/\varepsilon)$  with
$B(X/\varepsilon).$
\\

$(3):$ In case $G=F_{\sna},$ let $w\in
F_{\sna} \setminus X.$ Assume first that $w$ is
either the identity element of $F_{\sna}$ or has the
form $x^{-1}$ where $x\in X.$  Then
$$\overline{f}(w)\neq \overline{f}(y)=1 \ \forall y\in X$$  where
$\overline{f}:F_{\sna}\to \Z$ is the extension of the
constant function $f:(X,{\mathcal U}) \to \Z, \ f\equiv 1.$ Since $\Z$ is
discrete the set  $O:=\{z\in F_{\sna} |\ \overline{f}(z)\neq
1\}$ is clearly an open subset of $F_{\sna}$ and we have
$w\in O\subseteq X^{C}.$
 Let $k>1$ and $w=x_1^{t_1}\cdots x_k^{t_k} $ where
$t_i\in \Bbb Z\setminus \{0\}$ for every $1\leq i\leq k$   and
$x_i\neq x_{i+1}$ for every $ 1\leq i\leq k-1.$ Then there exists an
equivalence relation  $\varepsilon\in {\mathcal U}$ such that
 $$(x_i,x_{i+1})\notin \varepsilon \
\forall i\in \{1,2,\ldots, k-1\}.$$ Since $F(X/\varepsilon)$  is
algebraically free it follows that $[w]_{\varepsilon}\neq
[x]_{\varepsilon} \ \forall x\in X.$ Since $F(X/\varepsilon)$ is
discrete the set $U:=\{z\in F_{\sna} |\
f_{\varepsilon}(z)=[w]_{\varepsilon}\}$ is an open subset of
$F_{\sna}$ and we also have $w\in U\subseteq X^{C}.$ This
implies that $(X,{\mathcal U})$ is a closed subspace of $F_{\sna}.$ \\
For $G\neq F_{\sna}$ we may use the same modifications
appearing in the proof of $(2).$
\end{proof}

\begin{remark}
It is clear that if the universal morphism $i:(X,{\mathcal U})\to G$
is a uniform embedding, where $G$ is non-archimedean, then
$(X,{\mathcal U})$ is non-archimedean.
\end{remark}

\begin{lemma}\label{lem:finprec}
Let ${\mathcal U}$ be the discrete uniformity on a finite set $X.$
Then 
 $F^{\scriptscriptstyle Prec}_{\sna}$ algebraically is the free
group $F(X)$ over $X$.
\end{lemma}
\begin{proof}
It suffices to find a Hausdorff non-archimedean precompact group
topology $\tau$ on the abstract free group $F(X)$.
  Consider the group topology
$\tau$ generated by the filter base $\{N\vartriangleleft F(X): \
[F(X):N] < \infty\}.$  Clearly, $\tau$ is a non-archimedean
precompact group topology on $F(X).$ To see that $\tau$ is Hausdorff
recall that every free group is residually finite, that is, the
intersection of all normal subgroups of finite index is trivial.
\end{proof}

\begin{thm} \label{thm:fnprec} Let $(X,{\mathcal U})$ be a non-archimedean precompact uniform space and 
$$G\in \{F^{\scriptscriptstyle Prec}_{\sna},F_{\scriptscriptstyle Pro}\}.$$ Then:
 \ben
\item   The universal
morphism $i:(X,{\mathcal U}) \to G $ is a uniform embedding.
\item $F^{\scriptscriptstyle Prec}_{\sna}$ is 
algebraically free over $i(X)$.
\item  $i(X)$ is a closed subspace of $F^{\scriptscriptstyle Prec}_{\sna}$. 
\een
\end{thm}
\begin{proof}
$(1):$ $(X,{\mathcal U})$ is a uniform subspace of its compact zero
dimensional completion $(\widehat X,\widehat{\mathcal U})$. Consider the
compact group $\Z_2^{w(\widehat X)}$ where $w(\widehat X)$ is the
topological weight of $\widehat X.$ Then it is clear that $(\widehat
X,\widehat{\mathcal U})$ is uniformly embedded in $\Z_2^{w(\widehat X)}.$ Now,
since $\Z_2^{w(\widehat X)}$ is a profinite group then each
of the  universal morphisms  is a uniform embedding. \\
$(2):$ We use similar ideas to those appearing  in the proof of
Theorem \ref{thm:fnag}.2. This time due to the precompactness
assumption the set $X/ \eps$ is finite. By Lemma \ref{lem:finprec},
$F^{\scriptscriptstyle Prec}_{\sna}(X/ \eps)$ is algebraically free
over the set $X/ \eps$. Thus we may replace  $F_{\sna}(X/ \eps)$
with $F^{\scriptscriptstyle Prec}_{\sna}(X/ \eps),$ and also the
discrete topology on $\Z$ with  its Hausdorff topology generated by
all of its finite-index subgroups, to conclude  that
$F^{\scriptscriptstyle Prec}_{\sna}$ is
algebraically free over
$i(X)$.\\
$(3):$  Very similar to the proof of Theorem \ref{thm:fnag}.3. Just
observe that  $$O:=\{z\in F^{\scriptscriptstyle Prec}_{\sna} |\ \overline{f}(z)\neq 1\}$$ is an open subset of
$F^{\scriptscriptstyle Prec}_{\sna},$ since the group
topology on $\Z$ which we consider this time remains Hausdorff.
Moreover, the set
$$U:=\{z\in F^{\scriptscriptstyle Prec}_{\sna} |\
f_{\varepsilon}(z)\notin \{[x]_{\varepsilon}: x\in X\}\}$$ is also
an open subset of $F^{\scriptscriptstyle Prec}_{\sna},$ since $F^{\scriptscriptstyle Prec}_{\sna}(X/ \eps)$ is
Hausdorff and $\{[x]_{\varepsilon}: x\in X\}$ is finite.
\end{proof}

\section{Final non-archimedean group topologies}
\label{s:final} 

In this section 
 the
topological groups are not  necessarily Hausdorff. Recall that the
description of $F(X,{\mathcal U}),$ given by Pestov in \cite{pes85}
(see also Remark \ref{rem:pests}.1 below), was based on
final group topologies, which were studied by Dierolf and Roelcke
\cite[Chapter 4]{DR81}. Here we study final non-archimedean group topologies.

In the sequel we present a non-archimedean  modification of final
group topologies. The general structure
of final non-archimedean group topologies is then used to find
descriptions of the topologies for the free $\mathbf{NA}$ groups from Remark \ref{r:diag}.

We also provide a new description of the topology
 of $F^b(X,{\mathcal U}),$ the uniform free balanced group of a
 uniform space $(X,{\mathcal U}).$


\begin{defi}\label{def:finom}
Let $P$ be a group, 
$\alpha$ a  filter base on $P$  and $\Omega \subset \mathbf{TGr}$ an 
$SC$-variety.  Assume that there exists a  group topology $\tau$ on $P$
such that: \ben
\item  $(P,\tau)\in \Omega,$ and
\item the filter $\a$ converges to $e$ (notation: $\alpha \to e$) in  $(P,\tau).$
\een  Then among all group topologies on $P$ satisfying properties
$(1)$ and $(2)$ there is a finest one. We call it the \emph{$\Omega$-group
topology generated by $\alpha$} and denote it by
$\langle \alpha\rangle_{\Omega}.$
\end{defi}

\begin{defi}  \cite[Chapter 4]{DR81}
 If $P$ is a group and $(B_n)_{n\in \N}$ a sequence of subsets of
$P,$ let
$$[(B_n)]:=\bigcup_{n\in \N}
\bigcup_{\pi\in S_n} B_{\pi(1)} B_{\pi(2)}\cdots  B_{\pi(n)}.$$
\end{defi}

\begin{remark} \label{r:subgroup}
Note that if $(B_n)_{n\in \N}$ is a constant sequence such that
$$B_1=B_2=\cdots=B_n=\cdots =B$$   then $[(B_n)]=\bigcup_{n\in
\N}B^{n}.$ In this case we write $[B]$ instead of $[(B_n)].$
It is easy to see that if $B=B^{-1}$ then $[B]$ is simply the subgroup generated by $B.$
\end{remark}
\begin{lemma} \label{lem:fna}
Let $P$ be a non-archimedean topological group and $\mathcal{L}$ a
base of $N_{e}(P).$ Then the set $\{[B]: B \in \mathcal{L}\}$ is
also a base of $N_{e}(P).$
\end{lemma}
\begin{proof}
For every $ B \in\mathcal{L}$  we have $[B]\in N_{e}(P)$ since
$B\subseteq [B].$ Let $V\in N_{e}(P).$ We have to show that there
exists $B \in\mathcal{L}$ such that $[B]\subseteq V.$ Since $G$ is
non-archimedean, there exists an open subgroup $H$ such that
$H\subseteq V$ and $H\in N_{e}(P).$ On the other hand, $\mathcal{L}$
is a base of $N_{e}(P)$ and therefore there exists $B\in \mathcal{L}$
such that $B\subseteq H.$ From the fact that $H$ is a subgroup we
conclude
that $$[B]=\bigcup_{n\in \N} B^n\subseteq H\subseteq V.$$
%
%
%
\end{proof}

\begin{lemma} (Compare with \cite[Remark 4.27]{DR81}) \label{l:dar}
Let $P$ be a  group, $\alpha $ a filter base on $P$ and $\Omega$ an
$SC$-variety. Then:\ben
\item $$\mathcal{L}:=\Big\{\bigcup_{p\in P}(pA_{p}p^{-1}\cup
pA^{-1}_{p}p^{-1}):\ A_{p}\in \alpha \ (p\in P) \Big \}$$
 is also a filter base on $P$   and if $\langle \alpha\rangle_{\Omega}$ exists then $\langle \alpha\rangle_{\Omega}= \langle \mathcal{L}
\rangle_{\Omega}.$

\item If, in addition, all topological groups belonging to $\Omega$ are balanced then  $$\mathcal{M}:=\big\{\bigcup_{p\in P}(pVp^{-1}\cup
pV^{-1}p^{-1}):V\in \alpha \big \}$$ is a  filter base on $P$ and if
$\langle \alpha\rangle_{\Omega}$ exists then $\langle
\alpha\rangle_{\Omega}= \langle \mathcal{M} \rangle_{\Omega}.$  \een
\end{lemma}
\begin{proof} (1)  This follows from the fact that for every group
topology $\tau$ on $P$ such that $(P,\tau)\in \Omega$  the filter
base $\alpha$ converges to $e$ in $(P,\tau)$ if and only if
$\mathcal{L}$
converges to $e$ in $(P,\tau).$
Note that $\mathcal{L}$ satisfies the following
properties: \ben [(a)]
\item $ \forall A\in \mathcal{L}\ \ A=A^{-1},$
\item $ \forall A\in \mathcal{L}\ \forall p\in P  \  \exists B\in
\mathcal{L} \ \ pBp^{-1}\subseteq A.$ \een

(2) This follows from the fact that for every group topology $\tau$
on $P$ such that $(P,\tau)\in \Omega$  the filter base $\alpha$
converges to $e$ in $(P,\tau)$ if and only if $\mathcal{M}$
converges to $e$ in $(P,\tau).$
 Note that  $\mathcal{M}$ satisfies the following stronger
properties: \ben [(a*)]
\item $ \forall A\in \mathcal{M}\ A=A^{-1},$
\item $ \forall A\in \mathcal{M}\  \  \exists B\in
\mathcal{M} \  \forall p\in P  \ pBp^{-1}\subseteq A.$ \een

\end{proof}

\begin{lemma}\label{lem:fin}
Let $P$ be a group,  $\alpha$ a  filter base on $P$ and $\Omega$ an
$SC$-variety.  
\ben \item If $\Omega =\mathbf{NA}$ and $\alpha$ satisfies the
following 
 property: \ben[(a)]
\item $ \forall A\in \alpha \ \forall p\in P  \  \exists B\in
\alpha \ pBp^{-1}\subseteq A,$  \een
   then a base of $N_{e}(P,\langle
\alpha \rangle_{\Omega})$ is formed by the sets $[A],$  where $A\in
\alpha.$   \item If $$\Omega \in \{\mathbf{NA},\mathbf{NA_b},
\mathbf{AbNA}, \mathbf{BoolNA}\}$$
and $\alpha$ satisfies the stronger property
\ben [(a*)]
\item $\forall A\in \alpha \ \exists B\in \alpha \ \forall
p\in P \ \ pBp^{-1}\subseteq A,$ \een then the sets $[A],$ where
$A\in \alpha,$ constitute  a base of $N_{e}(P,\langle \alpha
\rangle_{\Omega}).$
\item If $\Omega =\mathbf{NA} \cap \mathbf{Prec}$ and $\alpha$ satisfies 
property  $(a^*)$ of $(2)$ then
$$\{N\lhd P| \ [P:N]<\infty \wedge \exists A\in \alpha \
[A]\subseteq N\}$$ is a local base at the identity element of
 $(P,\langle \alpha \rangle_{\Omega}).$
\item If $\Omega=\mathbf{SIN}$ and $\alpha$ satisfies the following
properties: \ben [(a*)]
\item $ \forall A\in \mathcal{M}\ A=A^{-1},$
\item $ \forall A\in \mathcal{M}\  \  \exists B\in
\mathcal{M} \  \forall p\in P  \ pBp^{-1}\subseteq A,$ \een
 then a base of $N_{e}(P,\langle \alpha
\rangle_{\Omega})$ is formed by the sets $[(A_n)],$ where $\forall
n\in \N \ A_n\in \alpha.$
 \een
\end{lemma}
\begin{proof}
$(1):$ Clearly $[A]^{2}\subseteq [A]$ and $[A]^{-1}=[A]\ \forall
A\in \alpha.$ Moreover,  for every $A\in \alpha$ and for every $p\in
P$ there exists $B\in \alpha$ such that $p[B]p^{-1}\subseteq [A].$
Indeed, we can use property $(a)$ to find $B \in \alpha$ such that
$pBp^{-1}\subseteq A.$ It follows that $p[B]p^{-1}\subseteq [A].$
This proves that
there exists a non-archimedean group topology $\mathcal{T}$ such
that $$\{[A]:A\in \alpha\}$$ is a base of $N_{e}(P,\mathcal{T}).$
Clearly, $\alpha$ converges to $e$ with respect to $\mathcal{T},$
and therefore $\mathcal{T} \subseteq \langle \alpha
\rangle_{\Omega}.$

Conversely,
let $\sigma$ be any non-archimedean group
topology on $P$ such that
$$\forall U\in N_{e}(P,\sigma) \ \exists A\in
\alpha \ A\subseteq U.$$

To prove that $\sigma \subseteq \mathcal{T}$, let $U\in
N_{e}(P,\sigma)$ be given. By Lemma \ref{lem:fna} there exists $V$
in $N_{e}(P,\sigma)$ such that $[V]\subseteq U,$ and, by the
assumption,   there exists a set $A\in  \alpha$ such that
$A\subseteq V$. Consequently, $[A]\subseteq [V]\subseteq U,$ which
proves $\sigma\subseteq \mathcal{T}.$\\ $(2):$ The proof of the
"balanced case" is quite similar. The only difference is 
the new condition
$$ \forall A\in \alpha \ \exists B\in \alpha \  \forall p\in P  \
\ pBp^{-1}\subseteq A,$$ which implies that the topology generated by the
sets $[A]$ is also balanced.\\
$(3):$ Precompact case:  clearly there exists a non-archimedean
precompact group topology $\mathcal{T}$ on $P$ such that   $$\{N\lhd
P| \ [P:N]<\infty \wedge \exists A\in \alpha \ [A]\subseteq N\}$$ is
a base of $N_{e}(P,\mathcal{T}).$ It is also trivial to see that
$\alpha$ converges to $e$ with respect to $\mathcal{T}.$
 Then, $\mathcal{T}$ is coarser
than $\langle \alpha \rangle_{\Omega}.$  Let $\sigma$ be any
precompact non-archimedean group topology on $P$ such that
$$\forall U\in N_{e}(P,\sigma) \ \exists A\in
\alpha \ A\subseteq U.$$

To prove that $\sigma \subseteq \mathcal{T}$, let $N\in
N_{e}(P,\sigma)$ be given. We can assume that $N$ is a finite-index
normal subgroup of $P.$  By Lemma \ref{lem:fna} there exists $V$ in
$N_{e}(P,\sigma)$ such that $[V]\subseteq U,$ and, by the
assumption,   there exists a set $A\in  \alpha$ such that
$A\subseteq V$. Consequently, $[A]\subseteq [V]\subseteq N$ which
proves $\sigma\subseteq \mathcal{T}.$\\
$(4):$ The proof is completely the same as the proof of
\cite[Proposition 4.28]{DR81}. Observe that from condition
$(b^*)$ it follows that the group topology, determined by the sets
$[(A_n)],$ is also balanced.
\end{proof}

\subsection{The structure of the free  $\mathbf{NA}$ topological
groups}\label{sub:str}

Let $(X,{\mathcal U})$ be a non-archimedean uniform space,
$Eq(\mathcal U)$ be the set of equivalence relations from ${\mathcal
U}$.
Denote by $j_2$ the mapping $(x,y)\mapsto x^{-1}y$ from $X^{2}$
to either $F(X),A(X)$ or $B(X)$ and by $j_{2}^{\ast}$ the mapping
$(x,y)\mapsto xy^{-1}.$ 

\begin{lemma}\label{lem:one} 
Let $(X,{\mathcal U})$ be non-archimedean and let 
$\mathcal{B}\subseteq Eq(\mathcal U)$ be a  base of ${\mathcal U}$.
 \ben
\item The topology of  $F_ {\sna}$  is the strongest among all non-archimedean
Hausdorff group topologies on $F(X)$ in which the filter
base
$$\mathcal{F}=\{j_{2}(V)\cup j_{2}^{\ast}(V)| \ V\in \mathcal{B} \}$$
converges to $e.$ \item For $$\Omega \in \{\mathbf{NA_b},\mathbf{NA}
\cap \mathbf{Prec}, \mathbf{AbNA}, \mathbf{BoolNA}\}$$ the topology
of $F_{\Omega}$ is
$\langle\mathcal{F}\rangle_{\Omega}$ where
$$\mathcal{F}=\{j_{2}(V)| \ V\in \mathcal{B}
\}.$$ \een

\end{lemma}
\begin{proof}
$(1):$ First recall that   $F_ {\sna}$ is
algebraically the abstract free group $F(X)$ (see Theorem
\ref{thm:fnag}.2). Let $\tau$ be a non-archimedean group topology on
$F(X).$ We show that $Id:(X,{\mathcal U})\to (F(X),\tau)$ is
uniformly continuous if and only if $\mathcal{F}$ converges to $e$
with respect to $\tau.$

The map $Id:(X,{\mathcal U})\to (F(X),\tau)$ is uniformly continuous
if and only if for every  $U\in N_{e}(F(X),\tau)$ there exists $V\in
\mathcal{B}$
 such that $$V\subseteq \tilde{U}=\{(x,y): x^{-1}y\in U \wedge
 xy^{-1}\in U\}.$$ The latter is equivalent to the following condition: there exists $V\in \mathcal{B}$
 such that $$j_{2}(V)\cup j_{2}^{\ast}(V)\subseteq U.$$ Thus, $Id:(X,{\mathcal U})\to (F(X),\tau)$ is uniformly continuous
 if and only if
$\mathcal{F}$ converges to $e$ with respect to $\tau.$ Clearly, the
topology of  $F_ {\sna}$  is a  non-archimedean
Hausdorff group topology on $F(X)$ in which the filter
base $\mathcal{F}$ converges to $e$. Moreover, for every non-archimedean
Hausdorff group topology $\tau$ on $F(X)$ in which $\mathcal{F}$
converges  to $e$, the map $Id:(X,{\mathcal U})\to (F(X),\tau)$ is uniformly
continuous. Therefore $Id:F_ {\sna} \to (F(X),\tau)$ is
uniformly continuous.
This completes the proof of $(1).$ \\ 
$(2):$ The proof  is very similar to the previous
case. 
 This time we can consider the filter base $\{j_{2}(V)| \ V\in \mathcal{B} \}$ instead of
  $\{j_{2}(V)\cup j_{2}^{\ast}(V)| \ V\in \mathcal{B} \}$ since
all the groups $F_{\Omega}$ are balanced.
\end{proof}
\begin{lemma}\label{lem:banone}
Let $(X,{\mathcal U})$ be a uniform space. For $\Omega =
\mathbf{SIN}$ the topology of $F_{\Omega}$ is
$\langle\mathcal{F}\rangle_{\Omega}$ where
$$\mathcal{F}=\{j_{2}(V)| \ V\in \mathcal{B}\}.$$
\end{lemma}
\begin{proof}
Use the same arguments as those appearing in the proof of Lemma
\ref{lem:one}.2.
\end{proof}
\begin{defi} \label{def:desc} \ben \item Following \cite{pes85}, for  every $\psi\in {\mathcal U}^{F(X)}$
let $$V_{\psi}:=\bigcup_{w\in F(X)}w(j_{2}(\psi(w))\cup
j_{2}^{\ast}(\psi(w)))w^{-1}.$$

\item  As a particular case in which every $\psi$ is a constant
function we obtain the  set
$$\tilde{\eps}:=\bigcup_{w\in F(X)}w(j_{2}(\eps)\cup
j_{2}^{\ast}(\eps))w^{-1}.$$ \een
\end{defi}
\begin{remark}\label{rem:sym}
Note that if $\eps\in Eq(\mathcal U)$  then $(j_{2}(\eps))^{-1}=j_{2}(\eps),
\ (j_{2}^{\ast}(\eps))^{-1}= (j_{2}^{\ast}(\eps))$ and
$$\tilde{\eps}=\bigcup_{w\in F(X)}w(j_{2}(\eps)\cup
j_{2}^{\ast}(\eps))w^{-1}=\bigcup_{w\in F(X)}wj_{2}(\eps)w^{-1}.$$ Indeed,
this follows from the equality
$wts^{-1}w^{-1}=(ws)s^{-1}t(ws)^{-1}.$

Note also that the subgroup $[\widetilde{\eps}]$ (see Remark \ref{r:subgroup}) generated by $\eps$ is normal in $F(X)$.
\end{remark}

The proof of the following lemma is straightforward.

\begin{lemma} \label{l:eps}
Let $\eps$ be an equivalent relation on a set $X$.
Consider the  function $f_{\varepsilon}:X \to
X/\varepsilon$. Then
   $ker(\overline{f_{\eps}})=[\tilde{\eps}]$, where $\overline{f_{\eps}}: F(X) \to F(X/\eps)$ is the induced onto homomorphism.
\end{lemma}

\begin{thm}\label{thm:desbal}
Let $(X,{\mathcal U})$ be a uniform space. Then
 $\{[(\tilde{\eps}_n)]:\ \eps_n\in  \mathcal U \ \forall n\in \N \}$ is a base of
$N_{e}(F^b).$
\end{thm}
\begin{proof}
 By Lemma \ref{lem:banone} the topology of $F^b$ is $\langle \mathcal{F}\rangle_{\Omega}$ where
$$\mathcal{F}=\{j_{2}(\eps)| \ \eps\in \mathcal{B}\}$$ and $\Omega$ is the class of all  balanced topological groups.
According to Lemma \ref{l:dar}.2 we have
$\langle\mathcal{F}\rangle_{\Omega}=\langle \mathcal{M}
\rangle_{\Omega}$, where
$$\mathcal{M}:=\Big\{\bigcup_{w\in F(X)}(wAw^{-1}\cup
wA^{-1}w^{-1}):\ A\in\mathcal{F}  \Big \}.$$ In particular,
$$N_{e}(F^b)=N_{e}(F(X),\langle\mathcal{M}\rangle_{\Omega}).$$
By the description of the sets $\tilde{\eps}$ in Definition
\ref{def:desc}.2 and  Remark \ref{rem:sym} we have $\mathcal{M}=\{\tilde{\eps}:
\eps\in \mathcal{B} \}$. Finally, use Lemma \ref{l:dar}.2 and Lemma
\ref{lem:fin}.4 to complete the proof.
\end{proof}
\begin{thm}\label{thm:nafin}
Let $(X,{\mathcal U})$ be non-archimedean and let 
$\mathcal{B}\subseteq Eq(\mathcal U)$ be a  base of ${\mathcal U}$.
Then: \ben
\item The  family (of subgroups) $\{[\mathcal{V}_{\psi}]:\ \psi\in \mathcal{B}^{F(X)}\}$ is a base of
$N_{e}(F_ {\sna}).$
\item
\ben
\item The  family (of normal subgroups)
$\{[\tilde{\eps}]:\ \eps\in  \mathcal{B}\}$ is a base of $N_{e}(F^b_ {\sna}).$
\item
The topology of $F^b_ {\sna}$ is the weak topology generated by the system of homomorphisms
$\{\overline{f_{\eps}}: F(X) \to F(X/\eps)\}_{\eps \in \mathcal{B}}$ on discrete groups $F(X/\eps)$.
\een
\een
\end{thm}
\begin{proof}
 $(1):$ By Lemma
\ref{lem:one}.1 the topology of $F_ {\sna}$ is
$\langle\mathcal{F}\rangle_{\Omega}$ where
$$\mathcal{F}=\{j_{2}(\eps)\cup j_{2}^{\ast}(\eps)| \ \eps\in \mathcal{B}
\}$$ and $\Omega$ is the class of all non-archimedean topological
groups. According to Lemma \ref{l:dar}.1
$\langle\mathcal{F}\rangle_{\Omega}=\langle\mathcal{L}\rangle_{\Omega}$
where
$$\mathcal{L}:=\Big\{\bigcup_{w\in F(X)}(wA_{w}w^{-1}\cup
wA^{-1}_{w}w^{-1}):\ A_{w}\in\mathcal{F}, \ \ w \in F(X) \Big \}.$$ In
particular,
$$N_{e}(F_
{\sna})=N_{e}(F(X),\langle\mathcal{L}\rangle_{\Omega}).$$ By the
description of the sets $\mathcal{V}_{\psi}$ in Definition \ref{def:desc}.1
and  Remark \ref{rem:sym} we have $$\mathcal{L}=\{\mathcal{V}_{\psi}: \ \psi\in \mathcal{B}^{F(X)}  \}.$$ Finally, use
Lemma \ref{l:dar}.1 and Lemma \ref{lem:fin}.1 to conclude the proof. \\
$(2.a):$ By Lemma \ref{lem:one}.2 the topology of $F^b_
{\sna}$ is $\langle \mathcal{F}\rangle_{\Omega}$
where
$$\mathcal{F}=\{j_{2}(\eps)| \ \eps\in \mathcal{B}\}$$ and $\Omega$ is the class of all non-archimedean balanced topological groups.
According to Lemma \ref{l:dar}.2 $
\langle\mathcal{F}\rangle_{\Omega}=\langle \mathcal{M}
\rangle_{\Omega}$ where
$$\mathcal{M}:=\Big\{\bigcup_{w\in F(X)}(wAw^{-1}\cup
wA^{-1}w^{-1}):\ A\in\mathcal{F}  \Big \}.$$ In
particular,
$$N_{e}(F^b_ {\sna})=N_{e}(F(X),\langle\mathcal{M}\rangle_{\Omega}).$$
By the description of the sets $\tilde{\eps}$ in Definition
\ref{def:desc}.2 and  Remark \ref{rem:sym} $\mathcal{M}=\{\tilde{\eps}:
\eps\in \mathcal{B} \}$. Finally, use Lemma \ref{l:dar}.2 and Lemma
\ref{lem:fin}.2.

$(2.b):$ Use $(2.a)$ and observe that $ker(\overline{f_{\eps}})=[\tilde{\eps}]$ by Lemma \ref{l:eps}.
\end{proof}


\begin{thm}\label{thm:desna}
 Let $(X,{\mathcal U})$ be non-archimedean and let 
$\mathcal{B}\subseteq Eq(\mathcal U)$ be a  base of ${\mathcal U}$.
\ben
\item (abelian case) For every $\eps\in \mathcal{B}$  denote by
$<\eps>$ the subgroup of $A(X)$ algebraically generated by the set
$$\{x-y\in A(X) : (x,y) \in \eps\},$$
then
$\{<\eps>\}_{\eps \in \mathcal{B} }$ is  a base of
$N_0(A_{\sna}).$
\item (Boolean
case) If $<\eps>$ denotes the subgroup of  $B(X)$  algebraically
generated by 
$$\{x-y\in B(X) : (x,y) \in \eps\},$$ then
 $\{<\eps>\}_{\eps \in \mathcal{B}}$ is  a base of
$N_0(B_{\sna}).$
\een
\end{thm}
\begin{proof}
$(1):$ By Lemma \ref{lem:one}.2 the topology of $A_
{\sna}$ is $\langle \mathcal{F}\rangle_{\sna}$ where
$$\mathcal{F}=\{j_{2}(\eps)| \ \eps\in \mathcal{B} \}.$$
Therefore,
$$N_0(A_{\sna})=N_0(A(X),\langle \mathcal{F}\rangle_{\sna}).$$
By Remark \ref{rem:sym} and Lemma \ref{lem:fin}.2 a base of
$N_0(A(X),\langle \mathcal{F}\rangle_{\sna})$ is formed by the sets
$[j_{2}(\eps)],$ where $\eps\in \mathcal{B}$ and
 $[j_{2}(\eps)]$ is the subgroup generated by $j_{2}(\eps).$ Since
 $\eps$ is symmetric we have
$$j_{2}(\eps)=\{y-x\in A(X) : (x,y) \in \eps\}=\{x-y\in A(X) :
(x,y)\in \eps\}.$$ The proof of $(2)$ is similar.
\end{proof}

\begin{remark}
Note that the system $\mathcal{B}$ in Theorem \ref{thm:nafin}.2 induces a naturally defined inverse limit
$\underleftarrow{\lim}_{\eps \in \mathcal{B}} F(X/\eps)$ (of discrete groups $F(X/\eps)$) which can be identified with
the complete group $\widehat{F^b_ {\sna}}$. Similarly, $\underleftarrow{\lim}_{\eps \in \mathcal{B}} \ A(X/\eps)$ and
$\underleftarrow{\lim}_{\eps \in \mathcal{B}} \ B(X/\eps)$ can be identified with the groups $\widehat{A_ {\sna}}$ and $\widehat{B_ {\sna}}$ respectively.
\end{remark}

\begin{thm} \label{t:metrizability}
Let $X:=(X,{\mathcal U})$ be a Hausdorff non-archimedean space and
$$
G \in \{F^{b}_{\sna}(X), A_{\sna}(X), B_{\sna}(X)\}.
$$
Then
\ben
\item 
$\chi (G)=w({\mathcal U})$  
and $w(G)=w({\mathcal U}) \cdot l({\mathcal U})$.
\item If $(X,d)$ is an ultra-metric space then  $G$ is an
ultra-normable group of the same topological weight as  $X$.
\een
\end{thm}
\begin{proof} (1)
One may assume in Theorems \ref{thm:nafin}.2 and \ref{thm:desna} that
$\mathcal{B}\subseteq Eq(\mathcal U)$ is a  base of ${\mathcal U}$
of  cardinality $w({\mathcal U})$. This explains $\chi (G)=w({\mathcal U})$.
Since $l^u(G)=l({\mathcal U})$ (Lemma \ref{l:UnLind}.3) we can conclude by Lemma \ref{l:UnLind}.1
that $w(G)=w({\mathcal U}) \cdot l({\mathcal U})$.

(2) Combine (1) and Lemma \ref{l:MetrBal}.
\end{proof}

\begin{remark} \label{rem:pests}
\ben
\item Note that Pestov  showed (see \cite{pes85})  that the set
$\{[(V_{\psi_{n}})]\},$ where $\{\psi_{n}\}$ extends over the family of
all possible sequences of elements from ${\mathcal U}^{F(X)},$ is a
base of $N_{e}(F(X,{\mathcal U})).$ \item Considering only the sequences
of constant functions,  we obtain the set $\{[(\tilde{\eps}_n)]:\
\eps_n\in  \mathcal U \ \forall n\in \N \}$
 which is a base of $N_{e}(F^b(X,{\mathcal U}))$
by Theorem \ref{thm:desbal}.\item Let $(X,{\mathcal U})$ be a
non-archimedean uniform space. One may take in $(1)$ only the constant
sequences and obtain the set $\{[\mathcal{V}_{\psi}]:\ \psi\in
\mathcal{U}^{F(X)}\}$ which is a base of $N_{e}(F {\sna})$ by Theorem \ref{thm:nafin}.1.
\item If, in addition, the functions $\psi$ are all constant we
obtain a base of $N_{e}(F^b_ {\sna})$ (see Theorem
\ref{thm:nafin}.2).
\item $\mathbf{NA}$ is closed under products and subgroups.
So $\mathbf{NA}$ is a \emph{reflective subcategory} (see for example, \cite[Section 9]{SS}) of $\mathbf{TGr}$.
For every topological group $G$ there exists a universal arrow $f: G \to r_{\sna}(G)$, where $r_{\sna}(G) \in \mathbf{NA}$.
For every uniform space $(X,\mathcal{U})$ the group $F_{\sna}(X,\mathcal{U})$ is in fact $r_{\sna}(G)$, where $G:=F(X,\mathcal{U})$.
\een
\end{remark}

\subsection{Noncompleteness of $A_{\sna}(X,{\mathcal U})$}

In this subsection we show  that $A_{\sna}(X,{\mathcal U})$ is never
complete for non-discrete $\mathcal{U}$. \begin{defi}\ben  \item Let
$w=\sum_{i=1}^n k_ix_i$ be a nonzero element of $A(X),$  where $n\in
\N,$ and  for all $1\leq i\leq n: \ x_i\in X$ and $ k_i\in
\Z\setminus \{0\}.$ Define the length of $w$ to be $\sum_{i=1}^n
|k_i|$ and denote it by $lh(w).$
\item The length of the zero element is  $0.$
\item For a non-negative integer $n$ we denote by $B_n$ the subset of $A(X)$
consisting of all words of length $\leq n.$\een
\end{defi}

\begin{lemma} \label{lem:fbaire}
For every $n\in \N$ the set $B_n$ is closed in $A_{\sna}.$
\end{lemma}
\begin{proof}
It suffices to show that for every word $w$ of length
 $>n$ there exists $\eps\in {\mathcal U}$ such
that $w+<\eps> \cap B_n=\emptyset.$ Since ${\mathcal U}$ is
Hausdorff there exists $\eps\in \mathcal U$ such that for every
$x\neq y \in supp(w)$ we have $(x,y)\notin \eps.$ It follows that
for every $(x,y)\in \eps $ we have either $lh(w+(x-y))=lh(w)$ or
 $lh(w+(x-y))=lh(w)+1.$ Therefore, $w+<\eps> \cap
B_n=\emptyset.$
\end{proof}

\begin{lemma} \label{lem:sbaire}
Let $(X,\mathcal{U})$ be a non-archimedean  non-discrete uniform
space. Then, for every $n\in \N, \ int(B_n)=\emptyset.$
\end{lemma}
\begin{proof}
Let $w\in B_n.$ Since  $(X,\mathcal{U})$ is non-discrete and
Hausdorff, every symmetric entourage $\eps\in \mathcal{U}$
contains infinitely many elements of the form $(x,y)$, where $x\neq
y.$ It follows that there exists $(x,y)\in \eps$ such that
$x\notin supp(w).$  Now, if $y\notin supp(w)$ then
$lh(w+x-y)=lh(w)+2.$ Otherwise, we have $lh(w+x-y)=lh(w)+2$ or
$lh(w+y-x)=lh(w)+2,$ and either one of these cases implies that $w+<\eps>\nsubseteq B_n.$
Therefore, by Theorem \ref{thm:desna}.1, $int(B_n)=\emptyset.$
\end{proof}

\begin{thm}
Let $(X,{\mathcal U})$ be a non-archimedean metrizable non-discrete
uniform space. Then $A_{\sna}$ is not complete.
\end{thm}
\begin{proof}
%
 By Theorem \ref{t:metrizability} $A_{\sna}$ is
metrizable. This group is indeed non-complete. Otherwise, by
 Baire Category Theorem $A_{\sna}$ is not the countable union of nowhere-dense closed
 sets. This contradicts the fact that $A_{\sna}=\cup_{n\in \N}
 B_n,$ where the
 sets $B_n$ are nowhere-dense and closed (see Lemmas  \ref{lem:fbaire} and
 \ref{lem:sbaire}).
\end{proof}

As a contrast recall that $A(X,{\mathcal U})$ is complete for every complete uniform space $(X,{\mathcal U})$
(which for ${\mathcal U}={\mathcal U_{max}}$ gives Tkachenko-Uspenskij theorem). See \cite[page 497]{AT}.


\section{Free profinite groups} \label{s:pro}

The free profinite groups (in several subclasses $\Omega$ of $\textbf{Pro}$) play a
major role
in several applications 
\cite{GL,RZ,FJ}.
By Lemma \ref{l:compl} the free profinite group $F_{\scriptscriptstyle Pro}$ 
 can be identified with the completion $\widehat{F^{\scriptscriptstyle Prec}_{\sna}}$
 of the free precompact $\mathbf{NA}$ group
$F^{\scriptscriptstyle Prec}_{\sna}$. Its description comes from the following result which for Stone spaces is a version of a known result in the theory of profinite groups.
See for example \cite[Prop. 3.3.2]{RZ}.


\begin{thm} \label{t:ProfCard}
Let $(X,{\mathcal U})$ be a non-archimedean precompact uniform space
and let
$\mathcal{B}\subseteq Eq(\mathcal U)$ be a
base of ${\mathcal U}$ of  cardinality $w({\mathcal U})$. Then:
\ben
\item The set
$ S:=\{H \lhd F(X): \ [F(X):H] < \infty, \ \exists \ \eps\in
\mathcal{B} \ \ \ [\tilde{\eps}]\subseteq H  \}
$
is a local base at the identity of $F^{\scriptscriptstyle Prec}_{\sna}.$
\item Let
$ G \in \{F_{\scriptscriptstyle Pro},
F^{\scriptscriptstyle Prec}_{\sna}\}.
$
Then $\chi(G)=w(G)=w({\mathcal U})=w(X)$.
In particular, $G$ is metrizable for every metrizable ${\mathcal U}$. 


 \een
\end{thm}
\begin{proof} $(1):$ By Lemma \ref{lem:one}.2 the topology of $F^{\scriptscriptstyle Prec}_{\sna}$  is $\langle
\mathcal{F}\rangle_{\Omega}$ where $$\mathcal{F}=\{j_{2}(\eps)| \
\eps\in \mathcal{B}\}$$ and $\Omega=\mathbf{NA} \cap \textbf{Prec}.$
According to Lemma \ref{l:dar}.2 the latter  coincides with $\langle
\mathcal{M}\rangle_{\Omega}$ where
$$\mathcal{M}:=\Big\{\bigcup_{w\in F(X)}(wAw^{-1}\cup
wA^{-1}w^{-1}):\ A\in\mathcal{F}  \Big \}.$$ In
particular,
$$N_{e}(F^{\scriptscriptstyle Prec}_{\sna})=N_{e}(F(X),\langle
\mathcal{M}\rangle_{\Omega}).$$ By the description of the sets
$\tilde{\eps}$ in Definition \ref{def:desc}.2 and  Remark \ref{rem:sym},
$\mathcal{M}=\{\tilde{\eps}: \eps\in \mathcal{B} \}$.  Finally, use Lemmas
\ref{l:dar}.2 and \ref{lem:fin}.3. \\
 $(2):$ 
 Since $U$ and $G$ are precompact we have $w({\mathcal U})=w(X)$ and $\chi(G)=w(G)$. 
So we have only to show that $\chi(G)=w({\mathcal U})$. Assume that
$|\mathcal{B}|=w(\mathcal{U})$. By the description of $S$ in $(1),$
it suffices to show that for every  $\eps\in B$ there are countably
many  normal finite-index subgroups of $F(X)$  containing
$[\tilde{\eps}]$. Consider the  function $f_{\varepsilon}:X \to
X/\varepsilon$.
  Note that $F(X/\varepsilon)$ is a free group with finite number of generators.
  It is well known that the set of all (normal) finite-index subgroups of $F(X/\eps)$ is countable.
 By the Correspondence Theorem there are countably many normal finite-index subgroups of $F(X)$ containing
   $ker(\overline{f_{\eps}}),$ where $\overline{f_{\eps}}: F(X) \to F(X/\eps)$ is the induced onto homomorphism.
Now in order to complete the proof recall that $ker(\overline{f_{\eps}})=[\tilde{\eps}]$ by Lemma \ref{l:eps}.
\end{proof}

Let $\Omega$ be an $\overline{S}C$-variety of groups. Following 
\cite{SS} let us say that two compact spaces $X$ and $Y$ are $\Omega$-\emph{equivalent} 
if their $\Omega$-free groups $F_{\Omega}(X)$ and $F_{\Omega}(Y)$ are topologically isomorphic. Notation: $X \cong_{\Omega} Y$.
In particular,
we have the classical concepts of $M$-equivalent (in the honor of Markov) and $A$-equivalent compact spaces
(for $\Omega=\mathbf{TGr}$ and $\Omega=\mathbf{AbGr}$, respectively).
For free compact (abelian) groups and the corresponding equivalence see \cite{HM}.

Similarly, we get the concepts of
 $\mathbf{NA}$-equivalent, $\mathbf{AbNA}$-equivalent and $\mathbf{Pro}$-equivalent compact spaces.
 The $\mathbf{Pro}$-equivalence is very rigid as the following remark demonstrates.

 \begin{remark} \label{r:Melnikov}
From  Melnikov's result (see \cite[Proposition
3.5.12]{RZ}) it follows that every free profinite group on a compact
infinite Stone space $X$ is isomorphic to the free profinite group of the
1-point compactification of a discrete space with cardinality
$w(X)$. So two infinite Stone spaces $X$ and $Y$ are
$\mathbf{Pro}$-equivalent if and only if $w(X)=w(Y)$.
This implies that there are $\mathbf{Pro}$-equivalent 
compact spaces which are not $M$ or $A$-equivalent.
\end{remark}


Note that if $X$ is the converging sequence space and $Y$ is the
Cantor set then $X \cong_{\mathbf{Pro}} Y$ by Remark
\ref{r:Melnikov}. On the other hand, $X \ncong_{\Omega} Y$, where
$\Omega =\mathbf{NA} \cap \textbf{Prec}$, because
$F^{\scriptscriptstyle Prec}_{\sna}(X)$ is countable in contrast to
$F^{\scriptscriptstyle Prec}_{\sna}(Y)$. It would be interesting to
compare $\Omega$-equivalences on Stone spaces (with the same weight
and cardinality) for different subclasses $\Omega$ of $\mathbf{NA}$.


\subsection{The Heisenberg group associated to a Stone
space and free Boolean profinite groups}\label{s:H}

To every Stone space $X$ we associate in \cite{MS1}
 the natural biadditive mapping
 \begin{equation} \label{1}
w: C(X,\Z_2) \times C(X,\Z_2)^* \to \Z_2
\end{equation}

Where $V:=C(X,\Z_2)$ can be identified with the discrete group (with
respect to  symmetric difference) of all clopen subsets in $X$.
Denote by $V^{\ast}:=\hom(V,\T)$ the Pontryagin dual of $V.$
Since $V$ is a Boolean group 
every character $V \to \T$ can be identified with a homomorphism into the unique
2-element subgroup $\Omega_2=\{1, -1\}$, a copy of $\Z_2$. The same
is true for the characters on $V^*$, hence the natural evaluation
map 
$w:V \times V^*  \to  \T$ ($w(x,f)=f(x)$) can be restricted
naturally to $V \times V^* \to \Z_2$. Under this identification
$V^{\ast}:=\hom(V,\Z_2)$ is a closed subgroup of the
compact group $\Bbb{Z}_2^{V}.$ In particular, $V^{\ast}$ is a Boolean profinite group.
Similar arguments show that, in general,
any Boolean profinite group $G$ is the Pontryagin dual of the discrete Boolean group $G^*$.

We prove in \cite{MS1} that for every
Stone space $X$ the associated Heisenberg type group
$H=(\Bbb{Z}_2\times V)\leftthreetimes V^{\ast}$ is always minimal.

This setting has some additional interesting properties. Note that
the natural evaluation 
map
$$\delta: X \to V^*, \ x \mapsto \delta_x, \ \ \ \delta_x(f)=f(x)$$
is a topological embedding into 
$V^*$, where $w(V^*)=w(X)$.
Moreover, if $X$ is a $G$-space then the induced action of $G$ on $V^*$ is continuous and $\delta$ is a $G$-embedding.

The set $\delta(X)$ separates points of $V$ via the
biadditive mapping $w$ in (\ref{1}). Hence the subgroup generated by
$\delta(X)$ in $V^*$ is dense.
  We can say more using
additional properties of Pontryagin duality. We thank M. Jibladze
and D. Pataraya for the following observation (presented here after
some simplifications).

\begin{remark} \label{Stone duality} (M. Jibladze and D. Pataraya)
The Boolean profinite group $V^*$ together with $\delta: X \to V^*$ in fact is the
free Boolean profinite group $B_{\scriptscriptstyle Pro}(X)$ of $X$. In order to see this let
$f: X \to G$ be a continuous homomorphism into a Boolean profinite
group $G$. Then $G$
is the Pontryagin dual of a discrete Boolean group $H$. That is,
$G=H^*$. Now consider the natural inclusion
$$\nu: G^*=\hom(G,\Z_2) \hookrightarrow V=C(X,\Z_2).$$ Its dual arrow
$\nu^*: V^* \to G^{**}$ can be identified with $\nu^*: V^* \to
G=G^{**}$ such that $\nu^* \circ \delta =f$. Such extension $\nu^*$
of $f$ is uniquely defined because the subgroup generated by
$\delta(X) \subset V^*$ is dense.
\end{remark}

\section{Surjectively universal groups}
\label{s:co-un}

We already proved in Theorem \ref{t:metrizability} that for metrizable uniformities the corresponding free balanced, free abelian and free Boolean non-archimedean groups
 are also metrizable. The same is true by Theorem \ref{t:ProfCard} for the free profinite group which can be treated as the free compact $\mathbf{NA}$ group over a uniform space.
These results 
 allow us to unify
 and strengthen some old and recent results about the existence and the structure of surjectively universal
 $\mathbf{NA}$ groups.

Let $\Omega$ be a class of
topological groups. We say that a topological group $G$ is
\emph{surjectively universal} (or, \emph{co-universal})
in the class $\Omega$ if
$G \in \Omega$ and every $H \in \Omega$ is isomorphic to a topological factor group of $G$.

\begin{remark} \label{rem:co-un}
We list some natural classes $\Omega$ containing surjectively universal groups:
\ben
\item (Ding \cite{Ding}) Polish groups. This result answers a long standing question of Kechris. 
\item (Shakhmatov-Pelant-Watson\cite{SPW})

\ben \item The class of all abelian Polish groups. More generally the class of all abelian complete groups with weight $\leq \kappa$.

\item The class of all balanced metrizable complete groups with weight $\leq \kappa$.\een

\item (See Gao \cite{GAO} for (a),(b),(c) and also Gao-Xuan \cite{Gao-Xuan} for (b),(c)) \ben

\item $\mathbf{NA}$ Polish groups.

\item $\mathbf{NA}$ abelian Polish groups.

\item $\mathbf{NA}$ balanced Polish groups.\een

\item (Gildenhuys-Lim \cite[Lemma 1.11]{GL} (see also 
\cite[Thm 3.3.16]{RZ})) 
Profinite groups of weight $\leq \kappa$. \een
\end{remark}



\begin{lemma} \label{l:co-unSp}
The Baire space $B(\kappa)=(\kappa^{\aleph_0},{\mathcal U})$ is co-universal in the class of all completely
metrizable non-archimedean uniform spaces with topological weight $\leq \kappa$.
\end{lemma}
\begin{proof}
(First proof)
By Shakhmatov-Pelant-Watson \cite{SPW} there exists a Lipshitz-1 onto open (hence, quotient)
map $B(\kappa) \to (X,d)$ for every complete bounded metric space $(X,d)$ with topological weight $\leq \kappa$.

(Second proof) By Ellis \cite{Ellis} for every complete ultra-metric
space $X$ and its closed subspace $Y$ there exists a uniformly
continuous retraction $r: X \to Y$ (which necessarily is a quotient
map by Lemma \ref{l:quot}).
\end{proof}

%

\begin{lemma} \label{l:quot} (see for example \cite[Cor. 2.4.5]{Eng})
If the composition $f_2 \circ f_1: X \to Z$ of continuous maps $f_1: X \to Y$ and $f_2: Y \to Z$ is a quotient map,
then $f_2: Y \to Z$ is a quotient map.
\end{lemma}

  \begin{thm} \label{t:surj}
  \ben
  \item $\widehat{F^b}_{\sna}(\kappa^{\aleph_0},{\mathcal U})$ is
  surjectively universal in the class of all balanced $\mathbf{NA}$ 
  metrizable complete groups with weight $\leq \kappa$.
  \item
$\widehat{A}_{\sna}(\kappa^{\aleph_0},{\mathcal U})$ is surjectively
universal in the class of all abelian $\mathbf{NA}$ metrizable
complete groups with weight $\leq \kappa$.
\item $\widehat{B}_{\sna}(\kappa^{\aleph_0},{\mathcal U})$ is surjectively universal in the class of all
  Boolean $\mathbf{NA}$ metrizable complete groups with weight $\leq \kappa$.
\een
\end{thm}
\begin{proof} We give a proof only for (1) because cases (2) and (3) are very similar.

By  Theorem \ref{t:metrizability},
${F^{b}}_{\sna}:={F^{b}}_{\sna}(\kappa^{\aleph_0},{\mathcal U})$ is metrizable. Hence, $\widehat{{F^{b}}}_{\sna}$ is metrizable, too.
Furthermore, the topological weight of $\widehat{{F^{b}}}_{\sna}$ is $\kappa$. Indeed,
$\kappa^{\aleph_0}$ topologically generates $\widehat{{F^{b}}}_{\sna}$. So, by Lemma \ref{l:UnLind}, we get
$$w(\widehat{{F^{b}}}_{\sna})=\chi(\widehat{{F^{b}}}_{\sna}) \cdot
l^u(\widehat{{F^{b}}}_{\sna}) = \aleph_0 \cdot l^u(\kappa^{\aleph_0},{\mathcal U})=\kappa.$$
Let $P$ be a
balanced $\mathbf{NA}$ metrizable complete group with topological weight $\leq \kappa$.
Then its uniformity $\rho$ is both complete (by definition) and non-archimedean (by Fact \ref{t:condit}.2).
By Lemma \ref{l:co-unSp} there exists a uniformly
continuous onto map $f: \kappa^{\aleph_0} \to (P,\rho)$ which is a
quotient map of topological spaces.
By the universal property of $\widehat{{F^{b}}}_{\sna}$ (Lemma \ref{l:compl})
 there exists a unique continuous homomorphism
$\widehat{f}: \widehat{{F^{b}}}_{\sna} \to P$ which extends $f$. 
That is, $f=\widehat{f} \circ i$, where $i: \kappa^{\aleph_0} \to \widehat{{F^{b}}}_{\sna}$ is the universal arrow.
Since $f: \kappa^{\aleph_0} \to P$ is a quotient map we obtain by
Lemma \ref{l:quot} that $\widehat{f}$ is a quotient map.
\end{proof}

If in (1) and (2) we assume that $\kappa=\aleph_0$ then we  in
fact deal with Polish groups. In this case we get very short proofs of the results mentioned
in Remark \ref{rem:co-un} (assertions 3.b and 3.c). A new
point in this particular case is that the corresponding universal
groups come directly as the 
free objects. Indeed, recall that by Lemma \ref{l:compl}
the complete groups $\widehat{F^b}_{\sna}, \widehat{A}_{\sna}, \widehat{B}_{\sna}$ in Theorem \ref{t:surj} are $\Omega_C$-free groups
for the corresponding classes.

Result  
(3) in Theorem \ref{t:surj} seems to be new.

\begin{quest} \label{3}
Let $\kappa > \omega$.
Is it true that there exists a \textbf{co-universal} space in the class of all complete non-archimedean uniform spaces
with $dw(X,{\mathcal U}) \leq  (\kappa, \kappa)$ ?
\end{quest}

 A positive solution for Question \ref{3} will imply, by Theorem \ref{t:metrizability} and the
approach of Theorem \ref{t:surj}, that there exists a co-universal group in the class of all non-archimedean balanced (abelian,
Boolean) groups with topological weight $\leq \kappa$.

The following theorem  is known in the theory of profinite groups. Here we provide a very short proof of the existence
of surjectively universal profinite groups of weight $\leq \mathfrak{m}$ using Hulanicki's theorem.
The first assertion can be derived from \cite[Lemma
1.11]{GL} (or \cite[Thm 3.3.16]{RZ}). Its version
for the case $\mathfrak{m}=\aleph_0$ goes back to Iwasawa. This case
was proved also in \cite{Gao-Xuan}. 

\begin{thm} \label{t:surjBIG}
\ben
\item 
For every infinite cardinal $\mathfrak{m}$ there exists a
surjectively universal group 
in the class $\mathbf{Pro}_{\mathfrak{m}}$ of all profinite groups of 
weight $\leq \mathfrak{m}$.
\item 
Every free profinite group $F_{\scriptscriptstyle Pro}(X)$ over
any infinite Stone space $X$ of weight $\mathfrak{m}$ is a surjectively universal group
in the class $\mathbf{Pro}_{\mathfrak{m}}$.
\een
\end{thm}
\begin{proof} (1) Let $X=\{0,1\}^{\mathfrak{m}}$.
By Theorem \ref{t:ProfCard}.3,
$w(F_{\scriptscriptstyle Pro}(X))=w(X).$
So, $F_{\scriptscriptstyle Pro}(X) \in \mathbf{Pro}_{\mathfrak{m}}$.
By the universal property of $F_{\scriptscriptstyle Pro}(X) $
it is enough to show that any $G \in \mathbf{Pro}_{\mathfrak{m}}$ is a continuous image of $X$.
It is obvious for finite $G$.
By a theorem of Hulanicki (see \cite[Thm 9.15]{HR}) every infinite
$G \in \mathbf{Pro}_{\mathfrak{m}}$ is homeomorphic to $\{0,1\}^{\chi(G)}=\{0,1\}^{w(G)}$.
Since $w(G) \leq {\mathfrak{m}}$, there exists a continuous onto map
$\phi: \{0,1\}^{\mathfrak{m}} \to G$.

(2) See Remark \ref{r:Melnikov}.
\end{proof}


\section{Automorphizable actions and epimorphisms in topological groups}
\label{s:aut} Resolving a longstanding principal problem introduced by K.
Hofmann, Uspenskij \cite{Us-epic} showed  that in the category of
Hausdorff topological groups epimorphisms need not have a dense
range. Dikranjan and Tholen  \cite{DT}  gave a rather direct proof
of this
important result of Uspenskij. Pestov gave 
a useful criterion \cite{Pest-epic, pest-wh} (Fact \ref{pest}) which
we use below in Theorem \ref{epic=dense}. This test is closely
related to the concept of the \emph{free topological $G$-group}
 of a (uniform) $G$-space $X$ introduced in \cite{Me-F}. We denote it by $F_G(X)$.
It is a natural $G$-space version of the usual free
topological group. Similarly to Definition \ref{d:FreeGr} one may define
\emph{$\Omega$-free (uniform) $G$-group} $F_{G,\Omega}(X,\mathcal{U})$.

A topological (uniform) $G$-space $X$ is said to
be \emph{automorphizable} if $X$ is a topological (uniform)
$G$-subspace of a $G$-group $Y$ (with its two-sided uniform structure).
Equivalently, if the universal morphism $X \to F_G(X)$ of $X$ into
the free topological (uniform) $G$-group $F_G(X)$ of the (uniform)
$G$-space $X$ is an embedding.

\begin{fact} \label{pest} \emph{(Pestov \cite{Pest-epic, pest-wh})}
Let $f: M \to G$ be a continuous homomorphism between Hausdorff
topological groups. Denote by $X:=G/H$ the left coset $G$-space,
where
$H$ 
is the closure of the subgroup $f(M)$ in 
$G.$ The following are equivalent: \ben
\item 
$f: M \to G$ is an epimorphism.
\item 
The free topological $G$-group $F_G(X)$ of the $G$-space $X$ is
trivial.
 \een
\end{fact}
%

Triviality in (2) means `as trivial as possible', that is, $F_G(X)$ is isomorphic to the
cyclic discrete group.

\begin{remark} \label{non-aut}
  Let $X$ be the 
   $n$-dimensional cube
$[0,1]^n$ or the $n$-dimensional sphere $\mathbb{S}_n$. Then by
\cite{Me-F} the free topological $G$-group $F_G(X)$ of the $G$-space
$X$ is trivial for every $n \in \N$, where $G=\Homeo(X)$ is the
corresponding homeomorphism group. So, one of the possible examples
of an epimorphism which is not dense can be constructed as the
natural embedding $H \hookrightarrow G$
where $G=\Homeo(\mathbb{S}_1)$ and $H=G_z$ is the stabilizer 
of a point $z \in \mathbb{S}_1$. 
 The same example serves as the original counterexample to the epimorphism problem
in the paper of Uspenskij \cite{Us-epic}.
\end{remark}


In contrast, for Stone spaces  we have:

\begin{fact} \label{p:aut}  \cite[Lemma 4.3.2]{MS1}
Every continuous action of a topological group $G$ on a Stone space
$X$ is automorphizable (in $\mathbf{NA}$).
Hence the canonical $G$-map $X \to F_G(X)$ 
is an embedding.
\end{fact}

Roughly speaking the action by conjugations of a subgroup $H$ of a
non-archimedean group $G$ on $G$ 
reflects all possible difficulties of the Stone actions.
Below, in Theorem \ref{t:AE}, we extend Fact \ref{p:aut} to a
much larger class of actions on 
non-archimedean uniform spaces, where $X$ need not be compact. This
will be used in Theorem \ref{epic=dense} which deals with epimorphisms into
$\mathbf{NA}$-groups.


%


\begin{defi} \label{d:quasib} \cite {MEG89, me-fr}
Let $\pi: G \times X \to X$ be an action and ${\mathcal U}$ be a uniformity
on $X$. We say that the action (or, $X$) is \emph{quasibounded} if
for every $\varepsilon \in {\mathcal U}$ 
there exist: $\delta \in {\mathcal U}$
and a neighborhood 
$O$ of $e$ in $G$ such that
$$(gx,gy) \in \varepsilon  \ \ \ \ \forall \ (x,y) \in \delta, \ g \in O.$$

We say that the action on the uniform space $(X,{\mathcal U})$ is $\pi$-\emph{uniform}
if the action is quasibounded and all $g$-translations are ${\mathcal U}$-uniformly continuous. Equivalently,
for every $\varepsilon \in {\mathcal U}$ and $g_0 \in G$
there exist: $\delta \in {\mathcal U}$
and a neighborhood 
$O$ of $g_0$ in $G$ such that
$$(gx,gy) \in \varepsilon  \ \ \ \ \forall \ (x,y) \in \delta, \ g \in O.$$
\end{defi}

For a given topological group $G$
denote by $\Unif^G$ the triples $(X,{\mathcal U},\pi)$ where  
$(X,{\mathcal U})$ is a uniform space and $\pi:G \times X \to X$ is a continuous $\pi$-uniform action.

It is an easy observation that if the action $\pi: G \times X \to X$
is ${\mathcal U}$-quasibounded and the orbit maps $\tilde{x}: G \to X$ are
continuous then $\pi$ is continuous.

It is a remarkable fact that a topological $G$-space $X$ is
$G$-\emph{compactifiable} if and only if $X$ is ${\mathcal U}$-quasibounded
with respect to some compatible uniformity ${\mathcal U}$ on $X$,  \cite{MEG89,Me-F}. This was the main motivation to introduce quasibounded actions.
This concept 
gives a simultaneous generalization of some important classes of actions on uniform spaces.

\begin{fact} \label{fac:puni}  \cite{MEG89,Me-F}
We list here some examples of actions from $\Unif^G$.

\ben
\item Continuous isometric actions of $G$ on metric spaces.
\item $\Comp^G \subset \Unif^G$. Continuous actions on compact spaces (with their unique compatible uniformity).
\item
$(G/H,{\mathcal U}_r) \in \Unif^G$. Let $X=G/H$ be the coset
$G$-space and ${\mathcal U}_r$ is the right uniformity on $X$ (which
is always compatible with the topology).
\item
Let $X$ be a $G$-group.
Then $(X,{\mathcal U}) \in \Unif^G$ for every ${\mathcal U} \in
\{{\mathcal U}_r,{\mathcal U}_l,{\mathcal U}_{l \vee r},{\mathcal
U}_{l \wedge r}\}$. \een
\end{fact}



Recall that the well known Arens-Eells
linearization theorem (cf.
\cite{AE}) 
 asserts that every uniform (metric) space can be
(isometrically) embedded into a locally convex vector space (resp.,
normed space). This theorem on isometric linearization of metric
spaces can be naturally extended to the case of non-expansive
semigroup actions provided that the metric is bounded \cite{Me-cs},
or, assuming only that the orbits are bounded \cite{Schroder}.

Furthermore, suppose that an action of a \emph{group} $G$ on a
metric space $(X,d)$
is only $\pi$-uniform 
in the sense of Definition \ref{d:quasib} (and not necessarily non-expansive).
Then again such an action admits an
isometric $G$-linearization on a normed space.

Here we give a non-archimedean
$G$-version of Arens-Eells type theorem 
for uniform group actions.
Note that we will establish an ultra-metric $G$-version in
Theorem \ref{t:AE} below. The assumption $(X,{\mathcal U}) \in \Unif^G$ in Theorems \ref{thm:auna} and \ref{t:AE}
is necessary by Fact \ref{fac:puni}.4.

\begin{thm} \label{thm:auna}
Let $\pi: G \times X \to X$ be a continuous 
action of a topological group $G$ on a non-archimedean 
uniform space $(X,{\mathcal U})$. If $(X,{\mathcal U}) \in \Unif^G$
then the induced action by automorphisms
$$\overline{\pi}: G \times B_{\sna}(X) \to B_{\sna}(X), \ (g,u)
\mapsto gu$$ is continuous 
and  $(X,{\mathcal U})$ is a uniform $G$-subspace of  $B_{\sna}(X).$
Hence, $(X,{\mathcal U})$ is uniformly $G$-automorphizable (in
$\mathbf{NA}$).
\end{thm}

\begin{proof} By Theorem
\ref{thm:desna}.2 $\{< \eps > \}_{\eps \in Eq(\mathcal U)}$ is a base
of $N_0(B_{\sna}).$ By Theorem \ref{thm:fnag},
$(X,{\mathcal U})$ is a uniform subspace of $B_{\sna}.$ It is easy to see that $(X,{\mathcal U})$ is in fact  a
uniform $G$-subspace of $B_{\sna}(X).$ We show now that the action
$\overline{\pi}$ of $G$ on $B_{\sna}(X)$ is quasibounded and
continuous. The original action on $(X,{\mathcal U})$ is
$\pi$-quasibounded and continuous. Thus, for every
   $\eps\in Eq(\mathcal U)$ and $g_0 \in G$, there exist:
  $\delta\in Eq(\mathcal U)$ and a neighborhood $O(g_0)$ of $g_0$ in $G$ such that
for every $(x,y) \in \delta$  and  for every $g \in O$ we have
 $$(gx,gy)\in \eps.$$ This implies that $$g<\delta> \subseteq <\eps>  \ \forall  g \in
 O,$$ which proves that $\overline{\pi}$ is quasibounded.
 The map $\iota: X \hookrightarrow B_{\sna}(X), \ x\mapsto \{x\}$ is a topological
$G$-embedding. Together with the fact that
$\iota(X)$ algebraically spans $B_{\sna}(X)$ this implies that the
orbit mappings 
 $G \to B_{\sna}(X), \ g \mapsto gu$ are continuous for all
$u \in B_{\sna}(X)$. So we can conclude that $\overline{\pi}$ is continuous (see 
the remark after Definition \ref{d:quasib}) and $B_{\sna}(X)$ is a
$G$-group. 
\end{proof}

\begin{remark}\label{rem:ubn}
Let $(X,\mu)$ be a non-archimedean uniform space and
$\pi: G \times X \to X$ be a continuous action such that $(X,\mu) \in \Unif^G$.
The lifted action of $G$ on $B_{\sna}$ is continuous as we proved in
Theorem \ref{thm:auna}. This implies that $B_{\sna}$ is the \emph{free topological $G$-group} of $(X,\mathcal{U})$ in the class 
$\Omega$ of non-archimedean Boolean $G$-groups. 
Similarly, one may verify that the same remains true for $F^b_{\sna}$, $A_{\sna}$,
$F^{\scriptscriptstyle Prec}_{\sna}$,
$F_{\scriptscriptstyle Pro}$.
The case of $F_{\sna}(X,{\mathcal U})$, however, is unclear.
\end{remark}

\vskip 0.3cm

Recall that the
sets $$\tilde{U}:=\{(aH,bH): bH\subseteq UaH\},$$ where $U$ runs
over the neighborhoods of $e$ in $G$, form a uniformity base  on
$G/H$. This uniformity ${\mathcal U}_r$ (called the right uniformity) is compatible with the
quotient topology (see, for instance, \cite{DR81}).

\begin{thm} \label{epic=dense}
Let $f: M \to G$ be an epimorphism in the category $\mathbf{TGr}$. 
Denote by $H$ the closure of the subgroup $f(M)$ in $G$. Then each of the following conditions implies that $f(M)$ is dense in $G$.
\ben
\item The coset uniform space $(G/H,{\mathcal U}_r)$ is non-archimedean.
\item $G \in \mathbf{NA}$.
\item \emph{(T.H. Fay \cite{Fay})} $H$ is open in $G$.
\een
\end{thm}
\begin{proof} (1)
We have to show that $H=G$. Assuming the contrary consider the
\emph{nontrivial} Hausdorff coset $G$-space $G/H$.
 By Fact \ref{fac:puni}.3 the natural continuous
left action $\pi:G\times G/H\to G/H $ is $\pi$-uniform. Hence, we can apply
Theorem \ref{thm:auna} to conclude that the nontrivial $G$-space
$X:=G/H$ is $G$-automorphizable in $\mathbf{NA}$. In particular,
we obtain that there exists a \emph{nontrivial} equivariant morphism of the $G$-space $X$ to 
a Hausdorff $G$-group $E$. This implies that the free topological
$G$-group $F_G(X)$ of the $G$-space $X$ is not trivial. By the
criterion of Pestov (Fact \ref{pest}) we conclude that $f: M \to G$
is not an epimorphism.

(2) By (1) it is enough to show that the right uniformity on $G/H$ is non-archimedean.
Since $G$ is $\mathbf{NA}$ there exists a local
base $\mathcal{B}$ at $e$ consisting of clopen subgroups. Then it is straightforward to show that
$\tilde{\mathcal{B}}:=\{\tilde{U}: U\in \mathcal{B}\}$ is a base
of the right uniformity of $G/H$ such that its elements are equivalence
relations on $G/H$.

(3) If $H$ is open then  $X=G/H$ is topologically discrete. The
discrete uniformity $\mathcal U$ is  compatible  and we have
$(X,{\mathcal U}) \in \Unif^G$. As in the proof of (1) we apply
Theorem \ref{thm:auna} and Fact \ref{pest}.
\end{proof}

Assertion (3) is just the theorem of Fay \cite{Fay} mentioned above
in Subsection \ref{s:ItroEpi}. In  assertion (1) it suffices to
assume that the universal non-archimedean image of the coset uniform
space $G/H$ is nontrivial.

Note that if
$H \to G$ is not an epimorphism in $\mathbf{NA}$ then, a fortiori,
it is not an epimorphism in $\mathbf{TGr}$.
With a bit more work we can  refine
assertion (2) of Theorem \ref{epic=dense} as follows.

\begin{thm} \label{epic-NA}
Morphism $f: M \to G$ in the category $\mathbf{NA}$ is an epimorphism in $\mathbf{NA}$ (if and) only if $f(M)$ is dense in $G$.
\end{thm}
\begin{proof} Assume that $X:=G/H$ is non-trivial where $H:=cl(f(M))$.
It is enough to show that there exists a $\mathbf{NA}$ group $P$ and a pair of distinct morphisms
$g,h: G\to P$ such that $g \circ f=h \circ f.$
Theorem \ref{thm:auna} says not only that the (nontrivial) $G$-space $X=G/H$ is $G$-automorphizable but also that
 it is $G$-automorphizable in $\mathbf{NA}$. By Theorem \ref{thm:auna}, $B_{\sna}(X) \in \mathbf{NA}$ is a $G$-group.
 Now choose the desired $P$ as the corresponding
 semidirect product of $B_{\sna}(X)$ and $G$.
 Since $\mathbf{NA}$ is closed under semidirect products we obtain that $P \in \mathbf{NA}.$
 According to the approach of \cite{Pest-epic} there exists a pair of distinct morphisms
$g,h: G\to P$ such that $g \circ f=h \circ f.$
\end{proof}

\section{Group actions on ultra-metric spaces and Graev type ultra-norms}
\label{sbs:gun}


\begin{lemma} \label{l:ext}
Let $f: X \to \R$ be a function on an ultra-metric space $(X,d)$.
There exists a one-point ultra-metric extension $X \cup \{b\}$ of $X$ such that $f$ is the distance from $b$ if and only if
$$
|f(x)-f(y)| \leq d(x,y) \leq max\{f(x) , f(y)\}
$$
for all $x,y \in X$.
\end{lemma}
\begin{proof}
The proof is an ultra-metric modification of the proof in \cite[Lemma 5.1.22]{Pest-book}  
which asserts the following.
Let $f: X \to \R$ be a function on a metric space $(X,d)$.
There exists a one-point metric extension $X \cup \{b\}$ of $X$ such that $f$ is the distance from $b$ if and only if
$$
|f(x)-f(y)| \leq d(x,y) \leq f(x) + f(y)
$$
for all $x,y \in X$.
\end{proof}

The following result in particular 
gives a Graev type extension for ultra-metrics on free Boolean groups $B(X)$.
To an ultra-metric space $(X,d)$ we assign the Graev type group
$B_{Gr}(X,d)$. The latter (ultra-normable) topological group is in fact $B_{\sna}(X,\mathcal{U}_d)$, where $\mathcal{U}_d$ is the
uniformity of the metric $d$.


\begin{thm} \label{t:AE}

Let $(X,d)$ be an ultra-metric space and  $G$  a topological group.
Let $\pi: G \times X \to X$ be a continuous 
action such that $(X,\mathcal{U}_d) \in \Unif^G$.
 Then there exists
 an ultra-normed Boolean $G$-group $(E,||\cdot||)$ and
an isometric $G$-embedding $\iota: X \hookrightarrow E$ (with closed $\iota (X) \subset E$) such that:
\ben
\item
The norm on $E$ comes from the Graev-type ultra-metric extension of $d$ to $B(X)$.
\item
The topological groups $E$  and $B_{\sna}(X,\mathcal{U}_d)$ (the free Boolean 
$\mathbf{NA}$ group) are naturally isomorphic.
\een
\end{thm}
\begin{proof}



Consider the \emph{free Boolean group}
$(B(X), +)$ over the set 
$X.$ The elements of $B(X)$ are finite subsets of $X$ and the group
operation $+$ is the symmetric difference of subsets. We denote
by $\textbf{0}$ the zero
element of $B(X)$ (represented by the empty subset of $X$).  Clearly, $u=-u$ for every $u \in B(X)$.
Consider the natural set embedding $$\iota: X \hookrightarrow B(X),
\  \iota(x)=\{x\}.$$ Sometimes we will identify $x \in X$ with
$\iota(x)=\{x\} \in B(X).$


Fix $x_0\in X$ and extend the definition of $d$ from $X$ to
$\overline{X}:=X\cup \{\textbf{0}\}$ by letting
$d(x,\textbf{0})=\max\{d(x,x_0),1\}.$

\vskip 0.2cm

\noindent \textbf{Claim 1:} $d:\overline{X}\times \overline{X}\to \R$
is an ultra-metric extending the original ultra-metric $d$ on $X$.
\begin{proof}
The proof can be derived from Lemma \ref{l:ext},  noting that
$$
|f(x)-f(y)| \leq d(x,y) \leq max\{f(x) , f(y)\}
$$
for $f(x):=\max\{d(x,x_0),1\}$.
\end{proof}

For every nonzero $u=\{x_1,x_2,x_3, \cdots, x_m\} \in B(X)$ define the {\it support of $u$} as $supp(u):=u$ if $m$ is even, $supp(u):=u\cup \{\textbf{0}\}$ if $m$ is odd.


By a \emph{configuration} we mean a finite subset of $\overline {X}
\times \overline {X}$ (finite relations). Denote by $\Conf$ the set
of all configurations.
We can think 
of  $\omega \in \Conf$ as a finite set of some pairs
$$
\omega=\{(x_1,x_2), (x_3, x_4), \cdots , (x_{2n-1}, x_{2n})\},
$$
where all $\{x_i\}_{i=1}^{2n}$ are (not necessarily distinct) elements of $\overline {X}$. If 
$x_i \neq x_k$ for all distinct $1\leq i,k \leq 2n$ then
$\omega$ is said to be \emph{normal}. 
For every $\omega \in \Conf$ the sum
$$
u:=\sum_{i=1}^{2n} x_i=\sum_{i=1}^{n} (x_{2i-1}-x_{2i})
$$
 belongs to $ B(X)$ and we say that $\omega$
\emph{represents} $u$ or, that $\omega$ is a
$u$-\emph{configuration}. Notation: $\omega \in \Conf(u)$. We denote
by $\Norm(u)$ the set of all normal configurations of $u$. If
$\omega \in \Norm(u)$ then necessarily $\omega \subseteq supp(u)
\times supp(u).$
It follows that  $\Norm(u)$ is a finite set for any given $u \in
B(X)$.

 Our aim is to define a Graev type ultra-norm $||\cdot||$
on the Boolean group
$(B(X), +)$ such that $d(x,y)=||x - y||, \ \forall x,y\in X$. 
For every configuration $\omega$ we define its $d$-\emph{length} by
$$\varphi(\omega)=\max_{1 \leq i \leq n} d(x_{2i-1},x_{2i}).
$$


\vskip 0.2cm

\noindent \textbf{Claim 2:} For every  nonzero element $u \in B(X)$
and every $u$-configuration
$$
\omega=\{(x_1,x_2), (x_3, x_4), \cdots , (x_{2n-1}, x_{2n})\}
$$
define the following elementary reductions: \ben
\item
Deleting a trivial pair $(t,t)$. That is, deleting the pair
$(x_{2i-1}, x_{2i})$ whenever $x_{2i-1}=x_{2i}$.
\item Define the \emph{trivial inversion at $i$} of $\omega$ as the replacement of $(x_{2i-1}, x_{2i})$
by the pair in the reverse order $(x_{2i}, x_{2i-1})$.
\item Define the \emph{basic chain reduction rule} as follows.
Assume that there exist distinct $i$ and $k$ such that $x_{2i}=
x_{2k-1}.$ We delete in the configuration $\omega$ two pairs
$(x_{2i-1}, x_{2i})$, $(x_{2k-1}, x_{2k})$ and add the  new
pair $(x_{2i-1}, x_{2k})$. \een

Then, in all three cases, we get again a  $u$-configuration.
Reductions (1) and (2) do not change the $d$-length of the
configuration. Reduction (3) cannot exceed the $d$-length.

\begin{proof}  Comes directly 
from the axioms of ultra-metric. In the proof of (3) observe that
$$x_{2i-1} + x_{2i} + x_{2k-1} + x_{2k}=x_{2i-1} +  x_{2k}$$
in $B(X)$. This ensures that the new configuration is again a
$u$-configuration.
\end{proof}

\vskip 0.2cm

\noindent \textbf{Claim 3:} For every  nonzero element $u \in B(X)$
and every $u$-configuration $\omega$
 there exists a normal $u$-configuration $\nu$ such that $\varphi(\nu) \leq \varphi(\omega)$.
\begin{proof}
Using Claim 2 after finitely many reductions of $\omega$ we get a
normal $u$-configuration $\nu$ such that $\varphi(\nu) \leq
\varphi(\omega)$.
\end{proof}

Now we can define the desired ultra-norm $||\cdot||$.
For every 
$u \in B(X)$ define
$$||u||=\inf_{\omega \in \Conf(u)} \varphi(\omega).$$

\vskip 0.2cm

\noindent \textbf{Claim 4:} For every nonzero $u \in B(X)$ we have
$$||u||=\min_{\omega \in \Norm(u)} \varphi(\omega).$$
\begin{proof}
By Claim 3 
 it is enough to compute $||u||$  via normal $u$-configurations. So, since $\Norm(u)$ is finite, we may replace $\inf$ by $\min$.
\end{proof}

\vskip 0.2cm

\noindent \textbf{Claim 5:} $||\cdot||$ is an ultra-norm on $B(X)$.
\begin{proof}
Clearly,  $||u|| \geq 0$ and $||u||=||-u||$
(even $u=-u$) for every $u \in B(X)$.
For the $\textbf{0}$-configuration  $\{(\textbf{0},\textbf{0})\}$ we obtain that 
$||\textbf{0}|| \leq d(\textbf{0},\textbf{0})=0,$ and so
$||\textbf{0}||=0$. 
Furthermore, if $u\neq \textbf{0}$ then for every $\omega \in
\Norm(u)$ and for each $(t,s)\in \omega$  we have $d(t,s)\neq 0.$ We
can use Claim 4 to conclude that  $||u||\neq 0$. Finally, we have to
show that
$$||u+v|| \leq \max \{||u|| , ||v|| \} \ \ \ \forall \ u,v \in B(X).$$
Assuming the contrary, there exist  configurations
$\{(x_i,y_i)\}_{i=1}^n,  \{(t_i,s_i)\}_{i=1}^m$ with $u=\sum_{i=1}^n
(x_i-y_i), \ v=\sum_{i=1}^m (t_i-s_i)$ such that
$$||u+v||> c:=\max\{\max_{1 \leq i \leq n} d(x_i,y_i),\max_{1 \leq i
\leq m} d(t_i,s_i)\}.$$ Since $\omega:=\{(x_1,y_1), \cdots,
(x_n,y_n), (t_1,s_1), \cdots ,(t_m,s_m)\}$ is a configuration of
$u+v$ with $||u+v|| > \varphi(\omega)=c$, we obtain a contradiction
to the definition of $||\cdot||$.
\end{proof}

\vskip 0.2cm \textbf{Claim 6:} $\iota: (X,d) \hookrightarrow
E:=(B(X),||\cdot||)$
is an isometric embedding, i.e.
$$||x-y||=d(x,y) \ \ \ \forall \ x,y \in X.$$
\begin{proof}
By Claim 4 we may compute the ultra-norm via normal configurations. For the element
$u=x-y \neq \textbf{0}$ the only possible \emph{normal }
configurations are $\{(x,y)\}$ or $\{(y,x)\}$. So $||x-y||=d(x,y).$
\end{proof}

One can prove similarly that $d(x,\textbf{0})=||x||.$ This
observation will be used in the sequel.

 \vskip 0.2cm

\textbf{Claim 7:} For any given $u \in B(X)$ with $u \neq
\textbf{0}$ we have
$$||u|| \geq \min\{d(x_i,x_k): \ \  x_i,x_k \in supp(u), \ x_i \neq x_k\}.$$
\begin{proof} Easily deduced from Claims 3 and 4.
\end{proof}


%


We have the natural group action $$\overline{\pi}: G \times B(X) \to
B(X), (g,u) \mapsto gu$$ induced by the given action $G \times X \to
X$. Clearly, $g(u+v)=gu +gv$ for every $(g,u,v) \in G \times B(X)
\times B(X)$.
So this action is by automorphisms. Clearly $g \textbf{0}
=\textbf{0}$ for every $g \in G.$ It follows that $\iota: X \to
B(X)$ is a $G$-embedding. \vskip 0.2cm \noindent \textbf{Claim 8:}
The action $\overline{\pi}$ of $G$ on $B(X)$ is quasibounded and
continuous.\begin{proof} The original action on $(X,d)$ is
$\pi$-quasibounded and continuous. Since $\textbf{0}$ is an isolated
point in $\overline{X}$ then the induced action on
$(\overline{X},d)$ is also continuous and quasibounded. Thus, for
every 
   $\varepsilon> 0$ and $g_0 \in G$, there exist: 
  $1\geq \delta>0$ and a neighborhood $O(g_0)$ of $g_0$ in $G$ such that
for every $(x,y) \in \overline{X}\times \overline{X}$ with
$d(x,y)<\delta$ and  for every $g \in O$ we have
 $$d(gx,gy) < \varepsilon. 
 $$ By the definition of $||\cdot||$ it is easy
to see that
$$||gu|| < \varepsilon  \ \ \ \ \forall \ ||u|| < \delta, \ g \in O.$$ This implies that the action
$\overline{\pi}$ of $G$ on $B(X)$ is quasibounded.
Claim 5 implies that $\iota: X \hookrightarrow B(X)$ is a
topological $G$-embedding. Since
$\iota(X)$ algebraically spans $B(X)$ and $B(X)$ is a topological group, it easily follows that every
orbit mapping 
 $G \to B(X), \ g \mapsto gu$ is continuous for every
$u \in B(X)$. Thus we can conclude that $\overline{\pi}$ is continuous (see 
the remark after Definition \ref{d:quasib}) and $B(X)$ is a
$G$-group. 
\end{proof}

\vskip 0.3cm

By Claims 5 and 6 (see also the remark after Claim 6)  $||\cdot||$
is an ultra-norm on  $B(X)$ which extends the ultra-metric $d$
defined on $\overline{X},$ and it can be viewed as a Graev type
ultra-norm. To justify this last remark and the assertion (1) of our theorem observe that $||\cdot||$
satisfies, in addition,  the following maximal property. \vskip 0.3cm
\noindent \textbf{Claim 9:} Let $\sigma$ be an ultra-norm on $B(X)$
such that
$$\sigma(x-y)=d(x,y)\ \forall x,y\in \overline{X}.$$ Then
$||\cdot||\geq \sigma.$
\begin{proof}
Let $u$ be a nonzero element of $B(X).$ By Claim 4 there exists a
normal configuration  $$ \omega=\{(x_1,x_2), (x_3, x_4), \cdots ,
(x_{2n-1}, x_{2n})\},
$$ such that $u=\sum_{i=1}^{n} (x_{2i-1}-x_{2i})$ and $||u||= \max_{1 \leq i \leq n} d(x_{2i-1},x_{2i}).$

Now, $\sigma$ is  an ultra-norm and we also have
$$\sigma(x-y)=d(x,y)\ \forall x,y\in \overline{X}.$$ By induction
we obtain that
$$\sigma(u)=\sigma(\sum_{i=1}^{n} (x_{2i-1}-x_{2i}))\leq \max_{1 \leq i \leq n}
\sigma(x_{2i-1}-x_{2i})= \max_{1 \leq i \leq n}
d(x_{2i-1},x_{2i})=||u||.$$
\end{proof}

 The proof of assertion (2) in view of the description of $B_{\sna}(X,\mathcal{U}_d)$ given by Theorem \ref{thm:desna},
 follows from the fact that for every $0< \eps<1$ the subgroup generated by
$$\{x-y \in B(X): d(x,y)<\eps \}$$  is precisely the $\eps$-neighborhood of $\mathbf{0}$ in $E$.

 Summing up we finish the proof of Theorem \ref{t:AE}.
\end{proof}

\begin{remark} \label{r:Gr=fr}
Similarly we can assign for an ultra-metric space $(X,d)$  the Graev type groups $A_{Gr}(X,d)$ and $F_{Gr}(X,d).$
Moreover, one may show (making use of Theorems \ref{thm:desna}.1 and \ref{thm:nafin}.2) that $A_{Gr}(X,d)=A_{\sna}(X,\mathcal{U}_d)$ and $F_{Gr}(X,d)=F^b_{\sna}(X,\mathcal{U}_d).$
\end{remark}


%

\begin{corol}
Every ultra-metric space is isometric to a closed subset of an
ultra-normed Boolean group.
\end{corol}

By a theorem of Schikhof \cite{Shik},
 every ultra-metric space can
be isometrically embedded
into a suitable non-archimedean valued field. Note that
every non-archimedean valued field is an ultra-normed abelian group.

%
%

\subsection{Continuous actions on Stone spaces}

Assigning to every Stone space $X$ the free profinite group $F_{\scriptscriptstyle Pro}(X)$
we get a natural functor $\gamma$ from the category of Stone $G$-spaces $X$ to the category of all profinite $G$-groups $P$
(see Remark \ref{rem:ubn} for the continuity of the lifted action).
This functor preserves the topological weight and
there exists a canonical $G$-embedding $j_X: X \hookrightarrow \gamma(X)=F_{\scriptscriptstyle Pro}(X)$ where
$F_{\scriptscriptstyle Pro}(X)$ is metrizable if (and only if) $X$ is metrizable (Theorem \ref{t:ProfCard}).

This means, in particular, that every Stone $G$-space is automorphizable in $\mathbf{Pro}$ and
the class of (metrizable) profinite $G$-groups is at least as complex as the class of (metrizable) Stone $G$-spaces.
In contrast, recall that  a compact $G$-space that is not a Stone space may not be automorphizable (see Remark \ref{non-aut}).

By Remark \ref{r:Melnikov} the profinite groups $j(X)$ and $j(Y)$ are topologically isomorphic for infinite Stone spaces $X,Y$ with the same weight.
However, if $X$ and $Y$ are $G$-spaces (dynamical systems) then the corresponding $G$-spaces $j(X)$ and $j(Y)$ need not are $G$-isomorphic.

\section{Appendix}
\label{sec:gra}



 By  Graev's  Extension Theorem (see \cite{GRA}),
  for every metric $d$ on  $X\cup \{e\}$ there exists a metric $\delta$ on $F(X)$ with the following
properties: \ben
\item $\delta$ extends $d.$
\item $\delta$ is a two sided invariant metric on $F(X).$
\item $\delta$ is maximal among all invariant metrics on $F(X)$
extending $d.$ \een


 Savchenko-Zarichnyi \cite{SZ} presented an ultra-metrization $\widehat{d}$ of the free group  of an ultra-metric space $(X,d)$ with
$diam(X)\leq 1.$ Gao \cite{GAO} studied Graev type ultra-metrics $\delta_u.$
The metrics $\delta_u, \widehat{d} $ satisfy properties $(1)$ and $(2)$ above. As to the maximal property $(3)$
one may show the following:

\begin{thm}\label{thm:grau}
 Let $d$ be an ultra-metric on $\overline{X}:=X\cup X^{-1}\cup \{e\}$ for which
the following conditions hold for every $x,y\in X\cup \{e\}$:
\ben\item $d(x^{-1},y^{-1})=d(x,y).$
\item $d(x^{-1},y)=d(x,y^{-1}).$
\een     Then: \ben [(a)]
\item The Graev ultra-metric $\delta_u$ is maximal among all
invariant ultra-metrics on $F(X)$ which extend the metric $d$
defined on $\overline{X}.$
\item If, in addition, $d(x^{-1},y)=d(x,y^{-1})=\max\{d(x,e),d(y,e)\}$
then $\delta_u$  is  maximal among all invariant ultra-metrics on
$F(X)$ which extend the metric $d$ defined on $X\cup \{e\}.$
\item If $diam(X)\leq 1$ and $d(x^{-1},y)=d(x,y^{-1})=1$ then $\delta_u=\widehat{d}.$ \een
\end{thm}

\bibliographystyle{amsalpha}

\begin{thebibliography}{10}

\bibitem{AE} R. Arens and J. Eells,
{\it On embedding uniform and topological spaces}, Pacific J. Math.,
{\bf 6} (1956), 397-403.

\bibitem{AT} A. Arhangel'skii and M. Tkachenko,  \emph{Topological groups and related structures},
v. 1 of Atlantis Studies in Math. Series Editor: J. van Mill. Atlantis Press, World Scientific, Amsterdam-Paris,
2008.

\bibitem{bk} H. Becker and A. Kechris, \emph{The Descriptive
     Set Theory of Polish group Actions},  London Math. Soc. Lecture Notes Ser. \textbf{232}, Cambridge Univ.
     Press,  1996.



 \bibitem{DR81} S. Dierolf and W. Roelcke, \emph{Uniform Structures in Topological Groups and their Quotients}, McGraw-Hill,
 New York, 1981.


 \bibitem{DT} D. Dikranjan and W. Tholen,
\emph{Categorical Structure of Closure operators with Applications
to Topology, Algebra and Discrete Mathematics}, Math. and
Appl., 346 Kluwer Academic, Dordrecht (1995).


\bibitem{D-s} D. Dikranjan,
\emph{Recent advances in minimal topological groups}, Topology Appl.,
\textbf{85} (1998), 53--91.

\bibitem{DM-s} D. Dikranjan and M. Megrelishvili, \emph{Minimality properties in topological groups}, in Recent Progress in General Topology III (to appear).

\bibitem{DTk} D. Dikranjan and M. Tkachenko,
\emph{Weakly complete free topological groups}, Topology Appl.,
\textbf{112} (2001), 259-287.

%
%
%


\bibitem{Ding} L. Ding,
\emph{On surjectively universal Polish groups},  Adv. Math.
\textbf{231} (2012), no. 5, 2557-2572.


\bibitem{DG07} L. Ding and S. Gao, \emph{Graev metric groups and Polishable subgroups}, Adv.
Math. {\bf 213} (2007) 887-901.

\bibitem{Ellis} R. Ellis,
\emph{Extending uniformly continuous pseudoultrametrics and uniform
retracts}, Proc. AMS, \textbf{30} (1971), 599-602.

\bibitem{Eng}
R. Engelking, {\em  General Topology\/},
Heldermann Verlag, Berlin, 1989.

\bibitem{Fay}
T.H. Fay, \emph{A note on Hausdorff groups}, Bull. Austral. Math.
Soc., \textbf{13} (1975), 117-119.


\bibitem{FJ}
M.D. Fried and M. Jarden, \emph{Field Arithmetic,} A Series of Modern Surveys
in Mathematics, v. 11, Third Edition, Springer, 2008.


\bibitem{GAO} S. Gao, \emph{Graev ultrametrics and surjectively universal non-Archimedean Polish groups},
Topology Appl., \textbf{160} (2013), no. 6, 862-870.

\bibitem{Gao-Xuan} S. Gao and M. Xuan, \emph{On non-Archimedian Polish groups with two-sided invariant
metrics}, preprint, 2012.

\bibitem{GL}
D. Gildenhuys and C.-K. Lim, \emph{Free pro-C-groups}, Math. Z.,
\textbf{125} (1972), 233-254.

\bibitem{GRA} M.I.  Graev, \emph{Theory of topological groups
I}, (in Russian), Uspekhi, Mat. Nauk {\bf 5} (1950), 2-56.


\bibitem{Hig} M. Higasikawa, \emph{Topological group with several
disconnectedness}, arXiv:math/0106105v1, 2000,  1-13.

\bibitem{HR} E. Hewitt and K.A. Ross, \emph{Abstract Harmonic Analysis
I}, Springer, Berlin, 1963.

\bibitem{HM} K.H. Hofmann and S.A. Morris, \emph{Free compact groups I: Free compact
abelian groups,}
\textbf{23} (1986) 41-64.


\bibitem{Husek}
M. Husek, \emph{Urysohn universal space, its development and Hausdorff's approach}, Topology Appl., \textbf{155} (2008),
1493-1501.

\bibitem{isbo} J. Isbell, {\it Uniform Spaces}, American
         Mathematical Society, Providence, 1964.

\bibitem{Katetov}
M. Katetov, \emph{On universal metric spaces}, in: Gen. Topology and its Relations to
Modern Analysis and Algebra VI, Proc. Sixth Prague Topol. Symp. 1986, Z. Frolik,
ed., Heldermann Verlag (1988), 323-330.


\bibitem{Kulpa}
W. Kulpa, \emph{On uniform universal spaces}, Fund. Math., \textbf{69} (1970), 243-251.


\bibitem{Lem84}
A.Yu. Lemin, \emph{Isosceles metric spaces and groups}, in: Cardinal
invariants and mappings of topological spaces, Izhevsk, 1984, 26-31.

\bibitem{Lem03}
A.Yu. Lemin, \emph{The category of ultrametric spaces is isomorphic
to the category of complete, atomic, tree-like, and real graduated
lattices LAT*}, Algebra Universalis, \textbf{50} (2003), 35-49.


\bibitem{Les86} A.Yu. Lemin and Yu. M. Smirnov, \emph{Groups of isometries
 of metric and ultrametric spaces and their subgroups}, Russian Math. Surveys  \textbf{41}  (1986), no.6, 213-214.

\bibitem{Mar} A.A. Markov, \emph{On free topological groups}, Izv. Akad. Nauk SSSR Ser. Mat. \textbf{9} (1945) 3-64.

\bibitem{MEG89}  M. Megrelishvili, \emph{Compactification and Factorization in the Category of $G$-spaces
}  in Categorical Topology and its Relation to Analysis, Algebra and
Combinatorics, J. Ad\'{a}mek and S. Maclane, editors, World
Scientific, Singapore (1989), 220-237.


\bibitem{Me-F} M. Megrelishvili, \emph{Free topological G-groups},
 New Zealand Journal of Mathematics, vol. \textbf{25} (1996), no. 1,
59-72.


\bibitem{me-fr} M. Megrelishvili,
{\it Fragmentability and continuity of semigroup actions}, Semigroup
Forum, {\bf 57} (1998), 101-126.




\bibitem{Me-cs} M. Megrelishvili,
{\it Compactifications of semigroups and semigroup actions},
Topology Proceedings, \textbf{31:2} (2007), 611-650.

\bibitem{MES01} M. Megrelishvili and T. Scarr, \emph{The equivariant universality
and couniversality of the Cantor cube}, Fund. Math. \textbf{167}
(2001), no. 3, 269--275.

\bibitem{MS1} M. Megrelishvili and M. Shlossberg, \emph{Notes on non-archimedean topological groups}, Top. Appl.,
\textbf{159} (2012), 2497-2505.

\bibitem{MS1arx} M. Megrelishvili and M. Shlossberg, \emph{Notes on non-archimedean topological groups}.
First version (2010) of arxiv.org/abs/1010.5987.

\bibitem{Mor} S.A. Morris, \emph{Varieties of topological
groups},  Bull. Austral. Math. Soc. \textbf{1} (1969), 145-160.

\bibitem{Num}
E.C. Nummela, \emph{On epimorphisms of topological groups,} Gen.
Top. and its Appl. \textbf{9} (1978), 155-167.

\bibitem{Num2} E.C. Nummela, \emph{Uniform free topological groups and Samuel
compactifications}, Topology Appl. \textbf{13} (1982), no. 1, 77-83.

\bibitem{pes85} V. G. Pestov, \emph{Neighborhoods of identity in free topological groups}, Vestn. Mosk. Univ. Ser. 1.
Mat., Mekh., No. 3 (1985), 8--10 .

\bibitem{Pest-Cat} V. G. Pestov, \emph{Universal arrows to forgetful functors from categories
of topological algebras}, Bull. Austral. Math. Soc., \textbf{48} (1993),
209-249.

\bibitem{Pest-epic} V.G. Pestov, \emph{Epimorphisms of Hausdorff groups
by way of topological dynamics}, New Zealand J. of Math. \textbf{26}
(1997), 257-262.

\bibitem{Pest-free} V.G. Pestov, \emph{On free actions, minimal flows, and a problem by
Ellis}, Trans. Amer. Math. Soc. \textbf{350} (1998), 4149-4175.

\bibitem{pest-wh}
V. Pestov, {\it Topological groups: where to from here?} Topology
Proceedings, {\bf 24} (1999), 421-502.
http://arXiv.org/abs/math.GN/9910144.

\bibitem{Pest-book}
V. Pestov, \emph{Dynamics of infinite-dimensional groups. The
Ramsey-Dvoretzky-Milman phenomenon.} University Lecture Series, {\bf 40}. American
Mathematical Society, Providence, RI, 2006.

\bibitem{RZ}
L. Ribes and  P.A. Zalesskii, \emph{Profinite Groups}, 2nd ed.,
Springer 2010.

\bibitem{RD}
W. Roelcke and S. Dierolf,
{\it Uniform structures on topological groups and their
quotients}, Mc Graw-hill, New York, 1981.

\bibitem{ro} A.C.M. van Rooij, \emph{Non-Archimedean Functional
Analysis}, Monographs and Textbooks in Pure and Applied Math.
\textbf{51},  Marcel Dekker, Inc., New York, 1978.

\bibitem{SZ} A. Savchenko and M. Zarichnyi,
\emph{Metrization of free groups on ultrametric spaces}, Top. Appl.,
\textbf{157} (2010), 724-729.

\bibitem{Shik}
W. H. Schikhof,  \emph{Isometrical embeddings of ultrametric spaces
into non- Archimedean valued fields}, Indag. Math. \textbf{46}
(1984), 51–53.

\bibitem{Sip} O.V. Sipacheva,
\emph{The topology of a free topological group}, J. Math. Sci. (N.
Y.) \textbf{131} (2005), no. 4, 5765-5838.


\bibitem{Schroder} L. Schr\"{o}der,
\emph{Linearizability of non-expansive semigroup actions on metric
spaces,} Topology Appl., \textbf{155} (2008), 1576-1579.
%

\bibitem{SPW}
D. Shakhmatov, J. Pelant and S. Watson, \emph{A universal complete
metric abelian group of a given weight}, in Topology with
Applications, Szeksz\'{a}rd, 1993, 431-439. Bolyai Mathematical
Studies 4. J\'{a}nos Bolyai Mathematical Society, Budapest, 1995.

\bibitem{SS}
D. Shakhmatov and J. Sp\v{e}vak,
\emph{Group-valued continuous functions with the topology of pointwise convergence,} Top. Appl., \textbf{157} (2010), 1518-1540.


\bibitem{Te} S. Teleman,
{\it Sur la repr\'esentation lin\'eare des groupes topologiques},
Ann. Sci. Ecole Norm. Sup. {\bf 74} (1957), 319-339.

\bibitem{Tk} M.G. Tkachenko,
\emph{On topologies of free groups}, Czech. Math. J., \textbf{34}
(1984), 541-551.

\bibitem{Us-Ur}
V.V. Uspenskii, \emph{On the group of isometries of the Urysohn universal metric space,} Comment. Math. Univ. Carolin. \textbf{31} (1990) 181-182.

\bibitem{Us-free} V.V. Uspenskij, \emph{Free topological groups of metrizable spaces},
Math. USSR Izvestiya, \textbf{37} (1991), 657-680.

\bibitem{Us-epic} V.V. Uspenskij, \emph{The epimorphism problem for Hausdorff topological groups,} Topology Appl.
\textbf{57} (1994), 287--294.





\bibitem{War}
S. Warner, \emph{Topological Fields}, North Holland, Mathematics
Studies, vol. \textbf{157}, North-Holland-Amsterdam, London, New
York, Tokyo, 1993.


\end{thebibliography}

\end{document}